\documentclass[a4paper,11pt]{amsart}
\usepackage[english]{babel}
\usepackage{color}
\usepackage{amsmath} 
\usepackage{amssymb} 
\usepackage{amsthm}
\usepackage{graphicx}
\usepackage[colorlinks=true, linkcolor=blue]{hyperref}

\usepackage{enumitem}   

\usepackage{lmodern}
\usepackage{microtype}

\usepackage{subfiles}

\usepackage{soul}

\usepackage[colorinlistoftodos,prependcaption,textsize=tiny]{todonotes}

\usepackage[margin=1in]{geometry}

\newcommand {\R} {\mathbb{R}} 

\newcommand {\T} {\mathbb{T}} 
\newcommand {\N} {\mathbb{N}}

\newcommand {\p} {\partial}

\newcommand{\bbT}{\mathbb{T}}

\newcommand{\bbZ}{\mathbb{Z}}

\newcommand{\calE}{\mathcal{E}}
\newcommand{\calF}{\mathcal{F}}

\newcommand{\calI}{\mathcal{I}}

\newcommand{\calR}{\mathcal{R}}

\newcommand{\eps}{\varepsilon}

\DeclareMathOperator{\dive }{div}

\theoremstyle{plain}

\newtheorem{theorem}{Theorem}

\newtheorem{remark}{Remark}
\newtheorem{lemma}{Lemma}[section] 

\newtheorem{prop}{Proposition}

\newtheorem{defi}{Definition}[section]

\date{\today}
\begin{document}
\title[Dissipative fluid equations around Couette flow]{Suppression of Fluid Echoes and Sobolev Stability Threshold for 2D Dissipative Fluid Equations Around Couette Flow}
\begin{abstract}
We study the Sobolev stability thresholds of 2d dissipative fluid equations around Couette flow on the domain $\T\times \R$. We prove a bound for general nonlinear interactions, which, for several fluid equations, reduces the proof of nonlinear stability to a linear stability analysis. We apply this approach to the examples of Navier-Stokes, Boussinesq and magnetohydrodynamic equations around Couette flow. This improves the Sobolev stability threshold for the Boussinesq equations around Couette flow and large affine temperature to $1/3$ and for the MHD equations around Couette flow and constant magnetic field to $1/3^+$. 
\end{abstract}

\author{Niklas Knobel}
\maketitle
\setcounter{tocdepth}{1}
\tableofcontents

\section{Introduction}
In fluid dynamics, a central question is the stability of shear flows in various fluid systems as dissipation tends to zero. This article develops a general framework to establish stability for a broad class of two-dimensional dissipative incompressible fluid equations near an affine shear flow in an infinite channel. We consider a velocity field $V: \ \R_{+}\times \bbT\times \R\to \R^2$ and a transported scalar or vector quantity  $\Phi $, that satisfy the governing equation
\begin{align}\begin{split}\label{Gen}
    \partial_t V +(V\cdot\nabla) V  + \nabla \Pi  &= \nu \Delta V +F_1[V,\Phi], \\
    \partial_t \Phi+ (V\cdot\nabla) \Phi &= \kappa \Delta\Phi +F_2[V,\Phi], \\
    \dive (V) &= 0,
        \end{split}\end{align}
with $(t,x,y)\in \R_{+}\times \bbT\times \R$.
Here, $F_1$ and $F_2$ represent interaction terms. We treat this equation as a general model encompassing various fluid systems, with the Navier–Stokes, %\eqref{NS}
 Boussinesq, %\eqref{bouss}
 and magnetohydrodynamic (MHD) %\eqref{MHD} 
equations serving as our primary examples. The Navier-Stokes equations correspond to the case when $\Phi=F_i=0$. The Boussinesq equations model a fluid with varying temperature and read 
\begin{align}\tag{B}\label{bouss0}\begin{split}
    \partial_t V +(V\cdot\nabla) V  + \nabla \Pi  &= \nu \Delta V -    e_2 \Theta , \\
    \partial_t \Theta+ (V\cdot\nabla) \Theta &= \kappa \Delta\Theta, \\
    \dive (V) &= 0,
\end{split}\end{align}
where $\Theta: \ \R_{+}\times \bbT\times \R\to \R$ is the temperature variation. The MHD equations model a conducting fluid in the presence of a magnetic field and read
\begin{align}
  \tag{MHD}\label{MHD0}
  \begin{split}
    \partial_t V + (V\cdot \nabla) V+ \nabla \Pi  &= \nu \Delta  V + (B\cdot\nabla) B, \\
    \partial_t B + (V\cdot\nabla) B &= \kappa\Delta   B +(B\cdot\nabla) V, \\
    \dive ( V)=\dive ( B ) &= 0,
  \end{split}
\end{align}
where $B: \ \R_{+}\times \bbT\times \R\to \R^2$ corresponds to the magnetic field.

For these models, we investigate the stability of stationary solutions corresponding to linear shear flow, called Couette flow
\begin{align*}
    V_s(x,y) &= \begin{pmatrix}
        y \\ 0 
    \end{pmatrix}
\end{align*}
combined with a stationary solution $\Phi_s(x,y)$.

The relationship between fluid flow stability and dissipation is a fundamental problem in fluid dynamics \cite{schmid2002stability}. As early as Reynolds’ pioneering work \cite{reynolds1883xxix}, the transition from laminar to turbulent pipe flow at low viscosities was investigated. Our primary goal is to \emph{quantify} this relationship through a concept we refer to as the \emph{stability threshold}. Since \eqref{Gen} represents a system, stability depends on the relevant parameters, specifically $\mu=\min\{\nu,\kappa\}$. The stability threshold in Sobolev spaces (cf. \cite{bedrossian2017stability}) is defined as the smallest  $\lambda \ge 0$ for which the following implication holds:
\begin{align*}
    \Vert V-V_s, \Phi-\Phi_s  \Vert_{H^N }\le\delta \mu^{\lambda }\vert \ln(\mu)\vert^{-\tilde \lambda } \quad&\implies \quad\text{Nonlinear stability},\\
    \Vert V-V_s, \Phi-\Phi_s \Vert_{H^N }\gg  \mu^{\lambda } \vert \ln(\mu)\vert^{-\tilde \lambda } \quad&\implies \quad\text{Possible instability},
\end{align*}
for some $\delta >0$ and $\tilde \lambda\ge 0 $. If $\tilde\lambda>0$, we say $\lambda^+$ is the threshold. To identify such thresholds, we establish a general estimate that bounds nonlinear interactions and possible echo effects via dissipation. This method reduces the nonlinear stability analysis to a carefully structured linear problem, requiring the introduction of adapted variables and a linear weight. 

Using this framework, we derive thresholds for various fluid models, including the Navier–Stokes, Boussinesq, and magnetohydrodynamic equations. In particular, our approach leads to the following improvements:
\begin{itemize}
    \item For the Boussinesq equations \cite{zhai2023stability}, the previous upper bound threshold of $1/2$ is improved to $1/3$;
    \item For the MHD equations \cite{Dolce,knobel2024sobolev,wang2024stability,jin2025stability}, the previous upper bound threshold of $1/2$ is improved to $1/3^+$ (under the condition $\nu^3\le \kappa\le \nu^{\frac 13 }) $;
    \item Our analysis also extends prior MHD results by incorporating vertical constant magnetic fields, whereas earlier works considered only horizontal ones.
\end{itemize}
Moreover, for the Navier–Stokes equations \cite{bedrossian2014enhanced, masmoudi2022stability, wei2023nonlinear}, our method recovers the previously established threshold of $1/3$ in a direct and unified manner. 

For a function $F$, we define its $x$ average as $F_==\int F dx$ and the non average part as $F_{\neq} =F-F_=$. The stability thresholds derived in this work can be summarized as follows:
\begin{theorem}[Thresholds of selected fluid systems]\label{thm:thres}
Let $N\ge 13$ and $0<\nu,\kappa\le \frac 1 {10}$, then we have the following stability results:
\begin{itemize}
    \item \textbf{Navier-Stokes equations:} There exists $\delta>0$ such that if the initial data satisfies
    \begin{align*}
        \Vert V_{in} -e_1 y  \Vert_{\dot H^1\cap \dot H^N}=\eps\le \delta \nu^{\frac 13 },
    \end{align*}
    then there exists a unique solution $V-e_1y\in C_t( \dot H^1\cap \dot H^N)$ such that 
    \begin{align*}
         \langle t\rangle \Vert V^x_{\neq} \Vert_{L^2}+\langle t\rangle ^2\Vert V^y \Vert_{L^2}\lesssim e^{-c\nu^{\frac 13} t } \eps
    \end{align*}
    for some constant $c>0$ independent of the dissipation. 
    \item \textbf{Boussinesq equations \eqref{bouss0} for $\nu=\kappa$:} Consider $\Theta_s(y)=\beta^2 y$ with $\frac 12 < \beta $. There exists   $\delta_\beta>0$ such that if
    \begin{align*}
       \Vert V_{in} -e_1 y, \Theta_{in} - \beta^2 y   \Vert_{\dot H^{\frac 12 }\cap \dot H^N}=\eps\le \delta_\beta  \mu^{\frac 13 },
    \end{align*}
    then there exists a unique solution $(V-e_1y,\Theta-\beta^2 y)\in C_t( \dot H^{\frac 12 }\cap \dot H^N)$ such that 
    \begin{align*}
         \langle t\rangle^{\frac 12 } \Vert V^x_{\neq},\Theta_{\neq}  \Vert_{L^2}+\langle t\rangle^{\frac 32} \Vert V^y \Vert_{L^2}\lesssim e^{-c\mu^{\frac 13} t } \eps
    \end{align*}
    for some constant $c>0$ independent of the dissipation. 
    
    \item \textbf{MHD equations \eqref{MHD0}:} Consider $B_s=\alpha$ with $\alpha  \in \R^2\setminus \{ 0 \}$, with either $\alpha_2 \neq 0$ or  $\nu^3\le \kappa\le \nu^{\frac 13 }$. Then for all $r>0$, there exists a $\delta_{\alpha, r } >0$ such that if
     \begin{align*}
         \Vert V_{in} -e_1 y, B_{in} -  \alpha    \Vert_{H^N}=\eps\le \delta_{\alpha, r }  \mu^{\frac 13 }\vert \ln (\mu)\vert^{-(1+r)}
     \end{align*}
    then there exists a unique solution  $(V-e_1y,B-\alpha)\in C_tH^N$ such that 
    \begin{align*}
          \Vert V^x_{\neq},B^x_{\neq}  \Vert_{L^2}+\langle t\rangle  \Vert V^y,B^y- \alpha_2 \Vert_{L^2}\lesssim e^{-c\mu^{\frac 13} t } \eps
    \end{align*}
    for some constant $c>0$ independent of the dissipation. 
\end{itemize}
\end{theorem}

\subsection{Approach and Main Result} 
To study perturbations around the Couette flow $V_s=ye_1$, it is natural to adopt a coordinate system moving along its characteristics. In these new coordinates, the perturbation variables are defined as 
 \begin{align*}
    v(t,x,y) &= V(t,x+ty,y)- V_s(y),\qquad \phi(t,x,y) = \Phi (t,x+ty,y)- \Phi_s(y).
\end{align*}
This change of variables modifies the spatial derivatives accordingly:
$$
\p_y^t := \p_y-t\p_x,\qquad  \nabla_t :=(\p_x,\p_y^t)^T, \qquad \Delta_t :=\p_x^2 +(\p_y^t)^2,\qquad \Lambda_t:=(-\Delta_t)^{\frac 12 }.
$$
We consistently use lowercase letters to denote quantities in the moving frame and uppercase letters for those in the original (stationary) frame.

To establish stability, we construct a suitable energy functional and employ time-dependent Fourier multipliers \cite{bedrossian2014enhanced, bedrossian2015inviscid}, along with symmetrization techniques \cite{bedrossian21,Dolce,knobel2025sobolev}, now standard in the field. Our analysis introduces novel weights and sharpens previously known stability thresholds. Moreover, we develop a more flexible framework that can be adapted to analyze thresholds in related systems, following an appropriate linear analysis. This is carried out in three steps:

\medskip

\noindent \textbf{Step 1: Linear dynamics.} We begin by analyzing the linearized system and constructing tailored unknowns $(f_1[\phi],q[v])$ that capture the essential dynamics. These variables are linearly stable and satisfy the approximate relations    
 $$
 f_1 \approx  \Lambda^{\tilde \gamma}_t \phi, \qquad v\approx  \Lambda_t^{-(1+\gamma)} \nabla^\perp_t q,
 $$
in suitable Sobolev norms. This choice simplifies the structure of the linear and nonlinear interactions.

\medskip

\noindent \textbf{Step 2: Main nonlinear term.} For the tailored variables, we analyze the nonlinear terms, whose leading-order contribution takes the form
\begin{align*}
    \p_t f_1 &\sim  \Lambda^{\tilde \gamma}_t (\nabla^\perp_t\Lambda_t^{-1-\gamma } q \cdot \nabla_t \Lambda^{-\tilde \gamma}_t f_1 )\sim \Lambda^{\tilde \gamma}_t (\nabla^\perp\Lambda_t^{-1-\gamma } q \cdot \nabla \Lambda^{-\tilde \gamma}_t f_1 ).
\end{align*}
To control this term, we introduce a carefully designed Sobolev weight $A$, leading to the following estimate for the nonlinear interaction:
\begin{equation}\label{NLint}
    \p_t \Vert A f_1\Vert_{L^2 }^2
    \sim  \langle A f_1, [ \Lambda^{\tilde \gamma}_tA,\Lambda_t^{-1+\gamma }\nabla^\perp q] \cdot \nabla \Lambda^{-\tilde \gamma}_tf_1\rangle=: NL .
\end{equation}

\medskip

\noindent\textbf{Step 3, Nonlinear energy estimates:} 
We then construct a total energy $\calE$ in terms of the tailored unknowns, satisfying the integral identity
\begin{align*}
     \calE(t)
    &\sim \calE(1)+\int_1^t   L + NL +LNL \ d\tau ,
    \end{align*}
where $L$ denotes linear contributions, $NL$ the leading-order nonlinear term from \eqref{NLint}, and  $LNL$ lower-order nonlinear interactions. We establish a priori energy estimates and, using local well-posedness, apply a bootstrap argument to show that $\calE(t)$ remains bounded for all time, provided the initial data are sufficiently small. The resulting threshold condition relates the parameter $\mu$ to the size of the initial perturbation.

The threshold in our approach is determined by estimating the main nonlinear contribution
\begin{align*}
    \int_1^t  NL  \ d\tau.
\end{align*}
To control this term, we use carefully constructed, time-dependent Fourier multipliers $A$ and $m$, which are rigorously introduced in Section \ref{seq:freq}. 
 The following theorem provides a precise bound on the main nonlinearity and constitutes the core estimate of our analysis.
\begin{theorem}\label{thm:main}
Let $0< \mu \le \frac 12 $, $N\ge 12$,   $0 \le  \tilde \gamma \le \gamma \le 1$, $r > 0$. Let $f_1,f_2: \T\times \R\to \R^d$, $d\in \N$ and $q: \T\times \R\to \R$, be functions satisfying the bootstrap bounds
\begin{align}
\begin{split}\label{boot1}
   \Vert  A f \Vert^2_{L^\infty_t L^2}  +c_1\int_1^t \mu \Vert \nabla_t A  f\Vert^2_{L^2} +\Vert \sqrt {\tfrac {\p_t m }m} A  f \Vert_{L^2}^2 d\tau&\le C\eps^2_f,\\
   \Vert  A q \Vert^2_{L^\infty_t L^2}  +c_1\int_1^t \mu \Vert \nabla_t A  q\Vert^2_{L^2} +\Vert \sqrt {\tfrac {\p_t m }m} A  q \Vert_{L^2}^2 d\tau&\le C\eps^2_q.
\end{split}
\end{align}
For some constants  $c_1, C>0$ and parameters $\eps_f,\eps_q>0$. Define the \emph{main nonlinearity} as   
\begin{align*}
   NL_{f_1,f_2,q}^\gamma &= \vert\langle A  f_1 , [A  \Lambda_t^{\tilde \gamma }, \nabla^\perp \Lambda_t^{-(1+\gamma)}q  ]  \cdot\nabla \Lambda_t^{-\tilde \gamma }  f_2)\rangle\vert.
\end{align*}
Then, there exists a constant $\tilde C_\gamma > 0$ such that 
\begin{align}\begin{split}\label{NLfest}
   \int_1^t  NL_{f_1,f_2,q}^\gamma d\tau &\le 
     \tilde C_\gamma \mu^{-\frac 13 } \eps^2_f \eps_q, \qquad\qquad \qquad \qquad \ \gamma >0,  \\
     \int_1^t NL_{f_1,f_2,q}^0 d\tau &\le \tilde C_0 \vert\ln(\mu)\vert^{1+r}    \mu^{-\frac 13 } \eps^2_f \eps_q, \qquad\qquad \ \gamma = 0. 
\end{split}\end{align}
\end{theorem}

\begin{remark}
    For $\gamma \in (0,1] $, the constant  $\tilde C_\gamma$ depends continuously on $\gamma$ and satisfies $\lim_{\gamma \to 0} \tilde C_\gamma= \infty$. Nonetheless, the bound for $\gamma=0$ remains finite due to the additional logarithmic loss in \eqref{NLfest}, which is absorbed into the constant $\tilde C_0< \infty$. The constant $\tilde C_\gamma$ depends also on $N$, $r$, $\tilde \gamma$,  $A$ and $m$ but is independent $\eps_f$, $\eps_q$ and $\mu$. 
\end{remark}

\begin{remark}
The structure of the main nonlinearity in Theorem~\ref{thm:main} appears in a wide range of fluid equations. For instance:
\begin{itemize}
\item In the Navier–Stokes equations, it corresponds to $\gamma=1$, $\tilde \gamma=0$;
\item In the Boussinesq equations, to $\gamma=\tilde \gamma=\frac12$;
\item In the MHD equations, to $\gamma=\tilde \gamma=0$.
\end{itemize}
This highlights the flexibility of our framework and the potential for broader applicability.
\end{remark}

\subsection{Literature Overview }\label{sec:literature}
For the Navier–Stokes equations around Couette flow, the first upper bound on the stability threshold in Sobolev spaces was established in \cite{bedrossian2014enhanced}, where the threshold was shown to be $\mu^{1/2}$. This was later improved to the sharp threshold $\mu^{1/3}$ in \cite{masmoudi2022stability}. Using Gevrey regularity, this threshold can be further refined \cite{bedrossian2014enhanced,li2022asymptotic} to be uniform in $\mu$, thanks to the control of the inviscid dynamics established in \cite{bedrossian2015inviscid,dengZ2019,dengmasmoudi2018}.

The Boussinesq equations model a fluid with temperature-induced density variations. Around Couette flow, an affine temperature profile modifies the dynamics and weakens inviscid damping. Linear asymptotics for varying Richardson numbers are analyzed in \cite{yang2018linear}. For large affine temperature gradients, nonlinear stability is shown in the dissipative case by Zhai and Zhao \cite{zhai2023stability}. For small temperature gradients, Zillinger proves nonlinear stability in \cite{zillinger2020enhanced}. Additional results in the dissipative setting can be found in \cite{zhang2023stability,niu2024improved,zillinger2020boussinesq,masmoudi2023asymptotic}. Without diffusion, Gevrey perturbations of size  $\eps$  remain stable up to time $t\sim \eps^{-2}$ \cite{bedrossian21}. For the viscous case with vanishing thermal dissipation, Gevrey 3 spaces are sufficient for stability \cite{masmoudi2022stability2, zillinger2021echo}.

The magnetohydrodynamics  equations describe the coupled evolution of the velocity and magnetic field in a conducting fluid. The magnetic field is advected by the velocity and exerts a Lorentz force in return. In the 3D setting around Couette flow and a constant magnetic field satisfying some Diophantine condition, Liss establishes stability when viscosity and resistivity are equal ($\kappa =\nu$) in \cite{liss2020sobolev}, and this is extended to the case $\mu \approx \kappa$ in \cite{rao2025stability} and to rational magnetic fields in \cite{wang2025stabilitythreshold3dmhd}. In $2d$, Dolce obtains a threshold in the regime $\nu \le \kappa \lesssim \nu^{\frac 13 }$ \cite{Dolce}. When $\kappa \le \nu$, low resistivity becomes destabilizing, leading to magnetic field growth at rate $\nu\kappa^{\frac 13 }$, as shown in \cite{knobel2024sobolev}. Wang and Zhang generalize these results to arbitrary  $\nu,\kappa>0$ in \cite{wang2024stability}. For anisotropic resistivity in the horizontal direction, stability is established in \cite{knobel2025sobolev}.

In the inviscid–resistive regime, Gevrey regularity ensures stability \cite{knobel2023echoes,zhao2024asymptotic}. In the ideal MHD case (no viscosity or resistivity), \cite{NiklasMHD2024} proves stability over times $t \approx \eps^{-1}$ for Gevrey perturbations of size $\eps$. The addition of viscosity shortens the lifespan of stability due to enhanced magnetic growth \cite{knobel2024nr}.

\subsection{Structure of the Article} 
 In Section \ref{sec:NLheu}, we provide a heuristic explanation for the nonlinear echo mechanism. Section \ref{seq:freq} introduces the relevant time-dependent Fourier multipliers, energies, and frequency sets. In Section \ref{sec:NL}, we prove Theorem \ref{thm:main} by decomposing the main nonlinearity into four components: reaction, transport, remainder, and average terms. Each component is estimated through further refined splittings.  Finally, in Sections \ref{sec:estbou} and \ref{sec:estmhd}, we apply Theorem \ref{thm:main} to derive the thresholds for the Boussinesq and MHD equations, respectively. In the appendix, Section \ref{sec:NS} we prove the threshold for the Navier-Stokes equations. In Section \ref{sec:linweight} we prove properties of linear weights.

\subsection{Notations and Conventions}
\label{sec:notation}
Let $v$ be a vector or scalar, then we define 
\begin{align*}
    \langle v \rangle &= \sqrt {1+\vert v \vert}. 
\end{align*}
We write $f\lesssim g$ if $f\le C g$ for a constant independent of all relevant constants. We denote $f\approx g$ iff $f\lesssim g$ and $g\lesssim f$. Constant is independent of all relevant parameters such as $\mu, \, \eps$ and $\delta$. 

We denote the Sobolev and Lebesgue spaces usually $L^p:= L^p(\T\times \R)$ and $H^s:= H^s(\T\times \R)$. For two function $f,g  \in L^2$ we denote the $L^2$ scalar product 
\begin{align*}
    \langle f,g \rangle &=\int_{\T\times \R} f(x,y) \overline g(x,y) d(x,y). 
\end{align*}
For time-dependent functions, we denote the spaces 
\begin{align*}
    L^p_tH^s &= L^p([0,t], H^s).
\end{align*}
For a function $f \in L^2$ we denote its Fourier transform as 
\begin{align*}
    \calF f(k,\eta ):=\hat f(k,\eta) &=\frac 1 {2\pi} \int_{\T\times \R} e^{i(kx+\eta y)}f(x,y) d(x,y).
\end{align*}
Here $k$ or $l$ usually denote the Fourier variable in $x$ and $\eta$ or $\xi$ denote the Fourier variable in $y$. For the sake of simplicity, we often omit writing the $\wedge$. For a function $f:\T\times \R\to \R$ we define its $x$ average and non $x$ average part as 
\begin{align*}
    f_=&= \int f dx, & 
    f_{\neq} &= f-f_=. 
\end{align*}

\section{Nonlinear Heuristics}\label{sec:NLheu}

 In this section, we provide a heuristic for why a threshold of $\mu^{\frac 13 }$ with an additional logarithmic term if $\gamma =0$ is optimal for stability. This heuristic is based on the heuristics done for the Euler equations established in \cite{bedrossian2015inviscid}, where so-called echoes lead to loss of Gevrey regularity. Instead of using Gevrey regularity, we consider small data such that the dissipation suppresses possible echoes. We consider the nonlinear term
\begin{align*}
    \p_t f &\approx   \Lambda^{\tilde \gamma}_t ((\Lambda_t^{-1-\gamma }\nabla^\perp q \cdot \nabla )\Lambda^{-\tilde \gamma}_t f ).
\end{align*}
This corresponds to the nonlinear term of 
\begin{itemize}
    \item Navier-Stokes, if $f= q =w $, $\gamma =1$ and $\tilde \gamma=0$
    \item Boussinesq equation, if $f=\zeta_2$ and $q=\zeta_1 $ with  $\gamma =\tilde \gamma=\frac 12 $
    \item MHD equation, if $f\in \{v,b \}$ and $q=\Lambda_t^{-1} \nabla^\perp_t \cdot v $ with  $\gamma =\tilde \gamma=0 $
\end{itemize}
For these equations, $q$ and $f$ are of the same size in Sobolev regularity. For this heuristic, we assume that $f$ is a scalar and we replace $q$ by $f$. The worst possible growth appears due to a high-low frequency interaction 
\begin{align*}
    \p_t f &\approx   \Lambda^{\tilde \gamma}_t ((\Lambda_t^{-1-\gamma }\nabla^\perp f^{hi}  \cdot \nabla) \Lambda^{-\tilde \gamma}_t f^{low} ),\\
   \p_t f(k,\eta) &\approx \sum_{l } \int d\xi \tfrac {\vert k,\eta-kt \vert^{\tilde \gamma } }{\vert l,\xi-lt  \vert^{\tilde \gamma } }\tfrac {\eta l -k \xi }{\vert k-l,\eta -\xi-(k-l)t \vert^{1+\gamma} } f^{hi}(k-l,\eta-\xi) f^{low}
   (l,\xi).
\end{align*}
For this nonlinear term, there can appear echo chains, see \cite{bedrossian2015inviscid,dengZ2019}. That means that around times $t_k = \frac \eta k $ the $k$ mode forces the $k-1$ mode, leading to growth. Since $t_{k} \le t_{k-1} $, this can yield iterative growth 
\begin{center}
\begin{tikzpicture}
    % Draw the nodes with labels
    \node at (1.8,0) (phi) {$f(k,\eta)$};
    \node at (4,0) (p2) {$f(k-1,\eta)$};
    \node at (6,0) (p3) {$\dots $};
    \node at (7.7,0) (p4) {$f(1,\eta).$};
    % Draw an arrow between the nodes
    \draw[->] (phi) -- (p2);
    \draw[->] (p2) -- (p3);
    \draw[->] (p3) -- (p4);

    % Draw a red quarter circle around the nodes
\end{tikzpicture}
\end{center}
If no dissipation is present, this iterative growth can lead to exponential loss of regularity and therefore requires spaces with exponential regularity, Gevrey spaces, for stability. For stability in Sobolev spaces, we use that dissipation suppresses this echo mechanism for small perturbations.

We model the high-low interaction by $f^{low} (-1,\xi)\approx  \eps e^{-\mu^{\frac 13 }t }\delta_{\xi=0} $.  We consider the $k$ mode acting on the $k-1$ mode and assume that $f(k,\eta)$ stays constant. We obtain the equation 
\begin{align*}
    \p_t f(k-1,\eta) 
    &\approx \eps e^{-\mu^{\frac 13 }t} \tfrac {\vert k,\eta-(k-1)t \vert^{\tilde \gamma } }{\langle t  \rangle^{\tilde \gamma }  }\tfrac {\eta  }{k^{1+\gamma} }\frac 1 {\langle t - \frac \eta k\rangle ^{1+\gamma} } f(k,\eta). 
\end{align*}
This term is large around resonant times  $t\approx \tfrac \eta k$, then  $\tfrac {\vert k,\eta-(k-1)t \vert^{\tilde \gamma } }{\langle t  \rangle^{\tilde \gamma }  }\approx 1$ and so
\begin{align*}
    \p_t f(k-1,\eta) 
    &\approx \eps e^{-\mu^{\frac 13 }t} \tfrac {t }{k^{\gamma} }\frac 1 {\langle t - \frac \eta k\rangle ^{1+\gamma} } f(k,\eta). 
\end{align*}
In particular, the interaction grows in time until the enhanced dissipation timescale $t \approx \mu^{-\frac 13 } $,  for which it holds that 
\begin{align*}
    \p_t f(k-1,\eta)
    &\approx \eps\mu^{-\frac 13 }  e^{-\mu^{\frac 13 }t} \tfrac {1 }{k^{\gamma} }\frac 1 {\langle t - \frac \eta k\rangle ^{1+\gamma} } f(k,\eta) .
\end{align*}
After integrating in time, we obtain
\begin{align*}
      f(k-1,\eta) 
    &\approx  f_{in} (k-1,\eta)+  \eps k ^{-\gamma } \mu^{-\frac 13 }  \vert f(k,\eta)\vert   \cdot \begin{cases}
        1, & \gamma >0, \\
        \vert \ln(\mu)\vert,  & \gamma =0 .
    \end{cases}
\end{align*}
Therefore, to suppress echoes we require 
\begin{align*}
    \eps &\lesssim   \mu^{\frac 13 }\begin{cases}
        1, &\gamma>0,\\
        \vert \ln (\mu)\vert^{-1},& \gamma =0,
    \end{cases}
\end{align*}
to ensure stability. In particular, for $\gamma>0$ this corresponds to the threshold of Theorem \ref{thm:main} while for $\gamma=0$ we obtain an additional $\vert \ln(\mu)\vert^r $ for $r>0$.

\section{Construction of Weights and Frequency Sets}\label{seq:freq}
In this section, we define explicitly which weights are admissible and define useful sets for the proof of Theorem \ref{thm:main}. We define the time-dependent Fourier weight
\begin{align*}
     A(k,\eta )&= \langle k ,\eta  \rangle^{N} e^{ c \mu^{\frac 1 3}t\textbf{1}_{k\neq 0}}m^{-1}(k,\eta),
\end{align*}
where 
\begin{align*}
     m(k,\eta )&=  m_\gamma(k,\eta ) M_L(k,\eta ) M_\mu(k,\eta ). 
\end{align*}
The weight $m_\gamma$ and $M_\mu$ are defined in the latter and $M_L$ is a $\gamma$-admissible multiplier in the sense of Definition~\ref{def:admiss}.

We define the \emph{main resonance weights}
\begin{align}\begin{split}\label{eq:m1weight}
    g_\gamma(s) &= c  \begin{cases}
        \tfrac1{\langle s \rangle^{1+\gamma}},& \gamma>0,\\
        \tfrac1{\langle s \rangle \vert \ln (1+\langle s \rangle)\vert^{1+r } },& \gamma=0,
    \end{cases}\\
    \tilde m_\gamma (k,\eta )&= \exp\left(\int_{-\infty}^t \sum_{n\neq 0 }  g_\gamma(\tau -\tfrac \eta n )\langle k-n \rangle ^{-3}d\tau \right),\\
     m_\gamma (k,\eta )&= \exp\left (\min(1,t\mu^{\frac 13 }) \int_{-\infty}^t \sum_{n\neq 0 }  g_\gamma(\tau -\tfrac \eta n )\langle k-n \rangle ^{-3}d\tau \right).
\end{split}\end{align}
The weight $m_\gamma$ is up to our knowledge a new construction and is motivated by the weight $m$ in \cite{knobel2024nr}. We use $m_\gamma$  to bound the reaction term for resonant frequencies in Section \ref{sec:NLheu}. The $\min(1, \mu^{\frac 1 3} t )$ is important to bound the commutator at the transport term, see the $T_m$ term in Subsection \ref{sec:trans}. We define the \emph{Enhanced dissipation weight}
\begin{align*}\begin{split}
\frac {\p_t M_\mu}{M_\mu}(t,k,\eta) &= 
        c^{-1} \tfrac {\mu^{\frac 1 3 }}{1+ \mu^{\frac 2 3 }\vert t-\frac \eta k \vert^{2} },\qquad k\neq 0,\\
        M_\mu(0,k,\eta )&=1,\\
        M_\mu(t,0,\eta)&=M_\mu (t,1,\eta)
\end{split}\end{align*}
This weight is used to obtain enhanced dissipation and is a standard weight, see for example \cite{liss2020sobolev}. For linear weights, we use the following definition  to define admissible weights: 

\begin{defi}\label{def:admiss}
    Let $N\ge 12$, $r>0$ and  $\gamma\in[0,1]$ the weight $M_L  $ is \emph{$\gamma$-admissible}  if there exist a uniform constant $C>0$ such that 
    \begin{itemize}
        \item[i)] \textbf{Monoton increasing} For almost all $t>0$,  $M_L$ satisfies 
        \begin{align*}
            \p_t M_L(t,k,\eta ) \ge 0.
        \end{align*}
        \item[ii)] \textbf{Boundedness}
        \begin{align*}
            C^{-1}\le M_L \le C.
        \end{align*}
\
        \item[iii)] \textbf{Commutator 1} let $\vert l,\xi\vert \ge 8\vert k-l,\eta-\xi\vert $, then  $M_L$ satisfies 
        \begin{align*}\begin{split}
        \tfrac {\vert   l,\xi\vert  }{t^{1+\gamma }}\vert M_L(k,\eta) -M_L(l,\xi)\vert &\le R(t, k,l,\eta,\xi) \cdot  \begin{cases}
            1, & \gamma >0,\\
            \vert \ln(\mu)\vert^{1+r},& \gamma =0,
        \end{cases}%\label{eq:MLc1}
    \end{split}
    \end{align*}
        for 
        \begin{align*}
            R(t, k,l,\eta,\xi)&=C \Big(1+ h(t) \mu^{-\frac 1 3 } + (t+\mu^{-\frac 13 })\sqrt {\tfrac {\p_t m_\gamma }{m_\gamma}(k,\eta )}\sqrt {\tfrac {\p_t m_\gamma }{m_\gamma}(l,\xi) }\\
        &\qquad\qquad \qquad  + \mu^{\frac 1 3 } (\vert l,\xi-lt \vert+\vert k,\eta-kt \vert) \Big)\langle t\mu^{\frac 13 }\rangle^n \vert k-l,\eta-\xi\vert^{N-5}
        \end{align*}
        with  $h(t) \ge 0$ such that $\Vert h\Vert_{L^1_t}\le  C $ and $n\in \N$. 
    \item[iv)] \textbf{Commutator 2} let $k\neq 0$ and $\vert k,\xi\vert \ge 8\vert \eta-\xi\vert $, then 
    \begin{align*}
        \vert M_L(k,\eta) -M_L(k,\xi)\vert &\le C \tfrac 1 k \vert \eta-\xi\vert \langle \eta-\xi\rangle ^{N-5}.%\label{eq:MLc2}
    \end{align*}

    \end{itemize}
\end{defi}
Note that Definition \ref{def:admiss} is multiplicative i.e. if $M_L^i$ are $\gamma_i$-admissible with $i\in\calI$ for a finite index set $\calI$, then $\prod_{i\in \calI}  M_L^i$ is $\min_i \gamma_i$ admissible. 

\begin{lemma}[Properties of $m_\gamma$]\label{lem:m1} Let $\gamma \ge 0$, $(k,\eta), (l,\xi) \in \bbZ\times \R $, then $m_\gamma$ satisfies 
     \begin{enumerate}[label=(\roman*)]
         \item $1\le m_\gamma\le \tilde m_\gamma  \le C$ for a finite constant $C=C(c,\gamma) $.
         \item The weights $\tilde m_{\gamma}$ and $m_{\gamma}$ satisfy
         \begin{align*}
             \frac {\p_t  m_\gamma} { m_\gamma} (k,\eta )&=c\textbf{1}_{t\le \mu^{-\frac 13 }} \mu^{\frac 13 }  F_\gamma(t,k,\eta)+\min(1,t\mu^{\frac 1 3}) \frac {\p_t  \tilde m_\gamma} { \tilde m_\gamma} (k,\eta )\\
             &\ge\min(1,t\mu^{\frac 1 3}) \frac {\p_t  \tilde m_\gamma} { \tilde m_\gamma} (k,\eta ),
         \end{align*}
         for $t\neq \mu^{-\frac1 3 } $ and $F_\gamma(t,k,\eta)$ monotone increasing in time such that  $0\le  F_\gamma \le    C <\infty$ for a constant $C=C(r,\gamma)$.
         \item It holds for $\gamma>0$

         \begin{align*}
              \frac 1 {\langle\frac{\eta}{k}- t\rangle^{1+\gamma} }\le  {\frac {\p_t  \tilde m_\gamma} { \tilde m_\gamma} (l,\xi)}\langle k-l,\eta-\xi \rangle^5
         \end{align*}
and for $\gamma =0$
         \begin{align*}
              \frac 1 {\langle\frac{\eta}{k}- t\rangle\vert \ln(1+\langle\frac{\eta}{k}- t\rangle)\vert^{1+r}}\le  {\frac {\p_t \tilde  m_0} { \tilde m_0} (l,\xi)}\langle k-l,\eta-\xi \rangle^5.
         \end{align*}
         Furthermore, for all $\gamma $ we obtain 
         \begin{align*}
             \sqrt{\frac {\p_t  \tilde m_\gamma} { \tilde m_\gamma} (k,\eta )}\lesssim  \sqrt{\frac {\p_t  \tilde m_\gamma} {\tilde  m_\gamma} (l,\xi)}\langle k-l,\eta-\xi \rangle^5,\\
             \sqrt{\frac {\p_t   m_\gamma} { \tilde m_\gamma} (k,\eta )}\lesssim  \sqrt{\frac {\p_t   m_\gamma} {  m_\gamma} (l,\xi)}\langle k-l,\eta-\xi \rangle^5. 
         \end{align*}
        
         \item It holds that 
         \begin{align*}
             \vert m_\gamma (k,\eta)- m_\gamma(l,\xi) \vert \lesssim \min(1, \mu^{\frac 1 3}t)  .
         \end{align*}
          \item For $\vert l,\xi \vert \ge  \vert k-l,\eta-\xi\vert $ it holds that 
         \begin{align*}
    \vert m_\gamma(k,\eta)-m_\gamma(k,\xi)\vert&\lesssim \tfrac {\vert \eta-\xi\vert}{\langle k\rangle  }. 
\end{align*}
         \item For $\vert l, \xi  \vert \ge \vert k-l, \eta-\xi\vert $ it holds that
        \begin{align*}
             \vert m_\gamma(k,\eta)&- m_\gamma(l,\xi ) \vert
             \lesssim \frac {\langle k-l,\eta-\xi\rangle } {\langle k\rangle } +\frac \eta {\langle k\rangle ^2 } \vert k-l\vert^3  \begin{cases}
              \left( \frac {\p_t  m_\gamma}{  m_\gamma}(k,\eta )\right)^{\tfrac 1 {1+\gamma } }, & \gamma>0,  \\
                  \vert \ln(\mu)\vert^{1+r}\frac {\p_t  m_\gamma }{ m_\gamma}(k,\eta ) + \mu, & \gamma =0.
             \end{cases}
    \end{align*}

     \end{enumerate}
 \end{lemma}
 The proof of Lemma \ref{lem:m1} can be found at the end of this subsection. 
\begin{lemma}[Properties of $M_\mu$]\label{lem:M2} Let $(k,\eta), (l,\xi) \in \bbZ\times \R $, then $M_\mu$ satisfies 
\begin{enumerate}[label=(\roman*)]
\item It holds that $\p_t M_\mu\ge 0$ and  $1\le M_\mu \le \exp(c^{-1} \pi )$. 
    \item For $k\neq 0 $ it holds that 
    \begin{align*}
    \mu^{\frac 1 3 } &\le  c\tfrac {\p_t M_\mu} {M_\mu}(k,\eta )  + 2c \mu(1+\vert t-\tfrac \eta k \vert^2).
\end{align*}
    \item For $k\neq 0$ and  $\vert k,\xi \vert \ge  \vert \eta-\xi\vert $ it holds the estimate 
         \begin{align*}
    \vert M_\mu(k,\eta)-M_\mu(k,\xi)\vert&\lesssim \tfrac {\vert \eta-\xi\vert}{k }. 
    \end{align*}
    \item For $\vert l,\xi \vert \ge  \vert k-l,\eta-\xi\vert $ it holds the estimate 
\begin{align*}
    \vert M_\mu(k,\eta)-M_\mu(l,\xi)\vert\lesssim \mu^{\frac 13 } (\vert l,\xi-lt\vert +t ) \langle \eta-\xi,k-l\rangle^2.
\end{align*}

     \end{enumerate}
     In particular, $M_\mu$ is admissible in the sense of Definition \ref{def:admiss}.

\end{lemma}
 The proof of Lemma \ref{lem:M2} can be found in the appendix.

\begin{lemma}\label{cor:L2}
Under the assumption of Theorem \ref{thm:main} we obtain the following estimates
\begin{align*}
    \Vert A f_{\neq }\Vert_{L^2L^2 }&\le c \mu^{-\frac 1 6} \eps_f, \\
    \Vert A q_{\neq }\Vert_{L^2L^2 }&\le c \mu^{-\frac 1 6} \eps_q.
\end{align*}
\end{lemma}
\begin{proof}
    This is a direct consequence of Lemma \ref{lem:M2} ii) and \eqref{boot1}. 
\end{proof}

\subsection{Frequency Sets}

In this subsection, we define useful sets for frequency splitting. We define the \textit{reaction, transport, remainder and average sets} as
\begin{align}
    \label{def:SR}S_R &=\left\{ ((k,\eta),(l ,\xi )) \in (\mathbb{Z}\times \mathbb{R})^2\, : \ \vert k-l,\eta -\xi \vert\ge 8 \vert l, \xi\vert, \ k \neq l \right\}, \\
    \label{def:ST}S_T &=\left\{ ((k,\eta),(l ,\xi )) \in (\mathbb{Z}\times \mathbb{R})^2\, : \ 8\vert k-l,\eta -\xi \vert\le  \vert l, \xi\vert, \ k \neq l  \right\},  \\
    \label{def:ScalR}S_\calR  &=\left\{ ((k,\eta),(l ,\xi )) \in (\mathbb{Z}\times \mathbb{R})^2 \, : \ \tfrac 1 8  \vert l, \xi\vert\le \vert k-l,\eta -\xi \vert\le  8 \vert l, \xi\vert, \ k \neq l \right\},\\
    \label{def:S=}S_=&= \left\{ ((k,\eta),(l ,\xi )) \in (\mathbb{Z}\times \mathbb{R})^2 \, :  \ k =l \right\},
\end{align}
the \textit{average reaction and  transport  sets} as
\begin{align}
    \label{def:SR=}S_{R,=}&=\{((k,\eta),(l ,\xi )) \in (\mathbb{Z}\times \mathbb{R})^2: \ \vert \eta-\xi \vert \ge  \vert k, \xi \vert, \ k=l\},\\
    \label{def:ST=}S_{T,=}&=\{((k,\eta),(l ,\xi )) \in (\mathbb{Z}\times \mathbb{R})^2: \  \vert \eta-\xi \vert <  \vert k,  \xi \vert, \ k=l\}.
\end{align}
The resonance cut multipliers are denoted as 
\begin{align}\begin{split}\label{eq:freqcut}
    \chi^{R}(k,\eta) &= \textbf{1}_{\frac 12 t \le \frac \eta k \le 2t },\\
    \chi^{NR}(k,\eta) &= 1-\chi^{R}(k,\eta).
    \end{split}
\end{align}

\subsection{Main Weight Estimates}\label{seq:Appweights}
In this subsection, we prove Lemma \ref{lem:m1}. To prove i) we use that for $\gamma>0$, 
\begin{align*}
    \int  \sum_{l\neq 0} \tfrac 1 {\langle\frac{\eta}{n}- \tau \rangle^{1+\gamma  }}\langle k-n \rangle ^{-3}d\tau \le (\pi+\tfrac 2 \gamma ) \sum_{n\neq 0}\langle k-n\rangle^{-3} \le \pi (\pi+\tfrac 2 \gamma ).
\end{align*}
For $\gamma=0 $ we use that $g_0(s)=\tfrac 1 {\langle s \rangle \vert \ln(1+\langle s \rangle)\vert^{1+r} } $ is integrable and thus 
\begin{align*}
    \int\sum_{n\neq 0} \tfrac 1 {\langle\frac{\eta}{n}- \tau \rangle\vert \ln(1+\langle \tau-\frac \eta n\rangle\vert^{1+r}  }\langle k-n \rangle ^{-3}d\tau < \infty, 
\end{align*}
independent of $\mu$. For ii), if $t \neq \mu^{-\frac 13 }$ we calculate 
\begin{align*}
    \p_t  m_\gamma (k,\eta )&= \p_t \exp\left ( c \min(1,t\mu^{\frac 13 }) \int_{-\infty}^t \sum_{n\neq 0 }  g_\gamma(\tau -\tfrac \eta n )\langle k-n \rangle ^{-3}d\tau \right)\\
     &=\big( c \mu^{\frac 13 } \textbf{1}_{t\le \mu^{-\frac 13 }}\underbrace{\int_{-\infty}^t \sum_{n\neq 0 }  g_\gamma(\tau -\tfrac \eta n )\langle k-n \rangle ^{-3}d\tau}_{=:F(t,k,\eta) }\\
     &\qquad +\min(1,t\mu^{\frac 13 })\underbrace{\sum_{n\neq 0 }  g_\gamma(\tau -\tfrac \eta n )\langle k-n \rangle ^{-3}}_{ \frac{\p_t\tilde m} {\tilde  m}(k,\eta)  } m_\gamma(k,\eta)
\end{align*}
where $F(t,k,\eta)$ is bounded and increasing. Property iii) is a consequence of $\langle t-\tfrac \xi k \rangle \le\langle t-\frac \eta k \rangle \langle \eta-\xi  \rangle $ for $\gamma>0$. For $\gamma=0$ we additionally use that for all $x,y\in \R$ it holds that  $\ln(1+\langle y\rangle )\langle x-y\rangle -  \ln(1+\langle x\rangle )\ge 0 $ and therefore 
\begin{align*}
    \ln(1+\langle t-\tfrac \eta k\rangle )\le \ln(1+\langle t-\tfrac \eta k\rangle )\langle \eta-\xi\rangle. 
\end{align*}
The iv) is a direct consequence of the definition of $m_\gamma$. For v) we estimate 
\begin{align*}
    \frac {\p_\eta  m_\gamma}{m_\gamma}(k,\eta) &=\min(1,\mu^{\frac 13 } t ) \p_\eta \int_{-\infty}^t  \sum_{n\neq 0}g(t-\tfrac \eta n)  \langle k-n \rangle^{-3} d\tau \\
    &=\min(1,\mu^{\frac 13 } t ) \sum_{n\neq 0}  \int^t_{-\infty} \p_\eta g(t-\tfrac \eta n)\langle k-n \rangle^{-3} d\tau \\
    &=-\min(1,\mu^{\frac 13 } t ) \sum_{n\neq 0}  \int^t_{-\infty} \tfrac 1 n \p_t g(t-\tfrac \eta n)\langle k-n \rangle^{-3} d\tau 
\end{align*}
and thus 
\begin{align*}
    \left\vert \frac {\p_\eta  m_\gamma }{m_\gamma}(k,\eta)\right\vert &\le \tfrac 1 k  \min(1,\mu^{\frac 13 } t ) \Vert g\Vert_{L^{\infty}}\sum_{n\neq 0}  \langle k-n \rangle^{-2} d\tau \lesssim \tfrac 1 k\min(1,\mu^{\frac 13 } t ) . 
\end{align*}
Thus, we infer v) by the mean value theorem. 

To prove vi), we split $\vert m_\gamma(k,\eta)- m_\gamma(l,\xi ) \vert\le \vert m_\gamma(k,\eta)- m_\gamma(l,\eta ) \vert+\vert m_\gamma(l,\eta)- m_\gamma(l,\xi ) \vert$. By v) $\vert m_\gamma(l,\eta)- m_\gamma(l,\xi ) \vert$ satisfies the bound of vi). Therefore, we establish the bound 
\begin{align*}
             \vert m_\gamma(k,\eta)- m_\gamma(l,\eta ) \vert
             &\lesssim \frac {\vert k-l\vert} {\langle k\rangle } +\frac \eta {\langle k\rangle^2 } \vert k-l\vert^3 \cdot \begin{cases}
              \left( \tfrac {\p_t  m_\gamma}{  m_\gamma}(k,\eta) \right)^{\tfrac 1 {1+\gamma } }, & \gamma>0,  \\
                  \vert \ln(\mu)\vert^{1+r}\tfrac {\p_t  m }{ m}(k,\eta ) + \mu,  & \gamma =0.
             \end{cases}
    \end{align*}
    For $\vert k-l\vert \ge \tfrac 18 \vert k\vert $ or $k-l=0$ the claim is clear, therefore we consider $1\le \vert k-l\vert \le\tfrac 18 \vert k\vert $. 
We use $\vert e^x -1 \vert \le \vert x \vert e^x $ to estimate 
\begin{align}\begin{split}\label{eq:mgg}
             &\quad \vert m_\gamma(k,\eta)- m_\gamma(l,\eta ) \vert \\
             &\lesssim   \min(1, \mu^{\frac 1 3} t )\left\vert \sum_{n\neq 0 }\int_{-\infty }^td\tau  g_\gamma(\tau-\tfrac \eta n) \langle k-n \rangle ^{-3}- g_\gamma(t-\tfrac \eta n)\langle l-n \rangle ^{-3}\right\vert \\
             &=\min(1, \mu^{\frac 1 3} t )\left\vert\sum_{n }\int_{-\infty }^t d\tau (\textbf{1}_{n\neq 0}g_\gamma(\tau-\tfrac \eta n) - \textbf{1}_{n+l-k\neq 0}g_\gamma(\tau-\tfrac \eta {n+l-k }))\langle k-n \rangle ^{-3}\right\vert \\
             &=\min(1, \mu^{\frac 1 3} t )\left\vert\sum_{n}\int_{s_1}^{s_2}d\tau g_\gamma(\tau)\langle k-n \rangle ^{-3}\right\vert 
   \end{split} \end{align}
with
    \begin{align*}
        s_1&=\begin{cases}
            t-\frac \eta {n+l-k }, & n+l-k \neq 0, \\
            -\infty,& n+l-k =0, 
        \end{cases}\\
        s_2&=\begin{cases}
            t-\frac \eta {n }, & n \neq 0, \\
            -\infty,& n =0.
        \end{cases}
    \end{align*}   
For  $\vert n-k\vert\ge \tfrac k {2\vert k-l \vert}$ we estimate the sum directly 
         \begin{align}
             \sum_{\substack{\vert n-k\vert\ge \tfrac k {2\vert k-l \vert}}}\int_{s_1}^{s_2} g_\gamma(\tau)\langle k-n \rangle ^{-3}d\tau
             &\le 2\frac {\langle  k-l\rangle } {\langle k\rangle } \sum_{n}  \langle k-n \rangle ^{-2} \Vert g \Vert _{L^1}\lesssim \frac {\langle  k-l\rangle } {\langle k\rangle } \label{eq:g1}
         \end{align}
which is consistent with vi). \\
For  $\vert n-k\vert< \tfrac k {2\vert k-l \vert}$, using $1\le \vert k-l\vert \le \tfrac 18 \vert k\vert $ we obtain $n \neq 0$ and $n+k-l \neq 0 $ and thus $s_1=t-\frac \eta {n+l-k }$ and  $s_2=t-\frac \eta {n }$. First, we establish the bound 
\begin{align}
    \int_{s_1}^{s_2} g_\gamma(\tau) d\tau \lesssim  \frac \eta {k^2 } \vert k-l\vert \left(\tfrac 1 {\langle s_1\rangle  }+\tfrac 1 {\langle s_2\rangle  }\right)\label{eq:intg}
\end{align}
by distinguishing between $0 \not \in  [s_1,s_2] $ and $0  \in  [s_1,s_2] $. Let $0 \not \in  [s_1,s_2] $, then  by mean value theorem there exists a $s\in [s_1,s_2]$
\begin{align*}
    \int_{s_1}^{s_2} g_\gamma(\tau)&= g(s) \eta \left(\frac 1 n -\frac 1{n+l-k}\right).
\end{align*}
Since $g$ is monotone increasing on $(-\infty , 0)$ and  monotone decreasing $(0,\infty)$ we obtain by $0 \not \in  [s_1,s_2] $ that 
\begin{align*}
    g(s)\le g(s_1)+ g(s_2 ) \lesssim \tfrac 1{\langle s_1 \rangle }+\tfrac 1{\langle s_2 \rangle }.
\end{align*}
From $\vert n-k\vert\le \frac k{2\vert k-l\vert }$ and $\vert k-l\vert\ge 1$ we infer $n\ge\frac12  k$ and therefore 
\begin{align*}
   \int_{s_1}^{s_2} g_\gamma(\tau) d\tau &\lesssim  \left(\tfrac 1{\langle s_1 \rangle }+\tfrac 1{\langle s_2 \rangle }\right) \frac \eta {k^2 } \vert k-l\vert,
\end{align*}
which is consitent with \eqref{eq:intg}.
Let $0 \in  [s_1,s_2] $, then we obtain 
\begin{align*}
    \langle s_2 \rangle \le \langle s_2-s_1 \rangle\lesssim   \frac \eta {k^2 } \vert k-l\vert.
\end{align*}
Therefore,
\begin{align*}
    \int_{s_1}^{s_2} g_\gamma(\tau)&\le\Vert g \Vert_{L^1}\le  \frac \eta {k^2 } \vert k-l\vert \tfrac 1 {\langle s_2\rangle  }\Vert g \Vert_{L^1} .
\end{align*}
which is consitent with \eqref{eq:intg} and therefore \eqref{eq:intg} holds. Using \eqref{eq:intg}  we estimate 
\begin{align*}
    \sum_{\substack{n\\ \vert n-k\vert\ge \tfrac k {2\vert k-l \vert}}}\int_{s_1}^{s_2} g_\gamma(\tau)\langle k-n \rangle^{-3}d\tau 
     &\lesssim   \frac \eta {k^2 } \vert k-l\vert \sum_{n}\left( \tfrac {\textbf{1}_{n\neq 0 }} {\langle t-\frac \eta n \rangle  }+\tfrac {\textbf{1}_{n+l-k\neq 0 }} {\langle t-\frac \eta {n+l-k} \rangle  }\right)\langle k-n \rangle^{-3}\\
     &\lesssim   \frac \eta {k^2 } \vert k-l\vert^4 \sum_{\substack{n\neq 0 }} \tfrac 1 {\langle t-\frac \eta n \rangle  }\langle k-n \rangle^{-3}.
\end{align*}

Now we distinguish between $\gamma >0$ and $\gamma =0 $.  For $\gamma >0$ we estimate  
\begin{align*}
    \sum_{\substack{n\neq 0 }} \tfrac 1 {\langle t-\frac \eta n \rangle  }\langle k-n \rangle^{-3}&\le \left(\sum_{\substack{n\neq 0 }} \tfrac 1 {\langle t-\frac \eta n \rangle^{1+\gamma}  }\langle k-n \rangle^{-3}\right)^{\frac 1{1+\gamma}}\cdot \left(\sum_{\substack{n\neq 0 }} \langle k-n \rangle^{-3}\right)^{\frac \gamma {1+\gamma}}\\
    &\lesssim \left(\sum_{\substack{n\neq 0 }} \tfrac 1 {\langle t-\frac \eta n \rangle^{1+\gamma}  }\langle k-n \rangle^{-3}\right)^{\frac 1{1+\gamma}}=\left( \tfrac {\p_t  \tilde m_\gamma}{  \tilde m_\gamma}(k,\eta )\right)^{\tfrac 1 {1+\gamma } }\\
\end{align*}
to estimate

\begin{align*}
     \sum_{\substack{n\\\vert n-k\vert\ge \tfrac k {2\vert k-l \vert}}}\int_{s_1}^{s_2} g_\gamma(\tau)\langle k-n \rangle^{-3}d\tau 
     & \lesssim  \frac \eta {k^2 }\vert k-l\vert^4 \frac 1 {\min(1,t\mu^{\frac 13 } )} \left( \tfrac {\p_t   m_\gamma}{   m_\gamma}(k,\eta )\right)^{\tfrac 1 {1+\gamma } }
\end{align*}
together with \eqref{eq:g1} and \eqref{eq:mgg} this yields iv) for $\gamma >0$. 

{For $\gamma =0$,} by distinguish between $\langle t-\frac \eta n \rangle\le \mu^{-1}$ and $\langle t-\frac \eta n \rangle\ge \mu^{-1}$ we estimate 
\begin{align*}
    \frac 1 {\langle t-\frac \eta n \rangle }&\lesssim  \frac {\vert \ln(\mu)\vert^{1+r} } {\langle t-\frac \eta n \rangle \vert \ln(1+\langle t-\frac \eta n \rangle)\vert^{1+r}  } + \mu .
\end{align*}
Therefore, 
\begin{align*}
    \frac \eta {k^2 } \vert k-l\vert^3\sum_{\substack{n\neq 0 }} \tfrac 1 {\langle t-\frac \eta n \rangle  }\langle k-n \rangle^{-3}&\lesssim 
    \frac \eta {k^2 }\left( \vert \ln(\mu)\vert^{1+r}\frac {\p_t \tilde m }{\tilde m}(k,\eta ) + \mu \right)\vert k-l\vert^3\\
    &\lesssim 
    \frac \eta {k^2 }\left( \frac {\vert \ln(\mu)\vert^{1+r}}{\min(1,t\mu^{\frac 13 })}\frac {\p_t  m }{ m}(k,\eta )  +  \mu \right)\vert k-l\vert^3
\end{align*}
together with \eqref{eq:g1} and \eqref{eq:mgg}  this  yields vi) for $\gamma =0$.

\section{Proof of the Main Theorem}\label{sec:NL}
 In this section, we prove Theorem \ref{thm:main}, i.e. we prove the nonlinear estimates \eqref{NLfest}. By slight abuse of notation, we omit the subscripts $f_1$ and $f_2$ and just write $f$. By Plancherel's theorem 
\begin{align*}
   NL_{f_1,f_2,a}^\gamma &= \vert\langle A  f , [A  \Lambda_t^{\tilde \gamma }, \nabla^\perp \Lambda_t^{-(1+\gamma)}  q]  \cdot \nabla \Lambda_t^{-\tilde \gamma }  f)\rangle\vert\\
   & \le \sum_{k,l  }\iint d(\xi,\eta ) (\textbf{1}_{S_R}+\textbf{1}_{S_T}+\textbf{1}_{S_\calR}+\textbf{1}_{S_=}) 
   \tfrac {\vert \eta l -k\xi \vert} {\vert k-l ,\eta-\xi-(k-l)t \vert^{1+\gamma}}\\
   & \qquad \qquad \qquad \cdot \left\vert \tfrac {\vert k,\eta-kt\vert^{\tilde \gamma }}{\vert l,\xi-lt\vert^{\tilde \gamma }} A(k,\eta)-A(l,\xi)\right\vert \vert A f\vert(k,\eta ) \vert  f\vert(l,\xi )\vert  q\vert(k-l,\eta-\xi )\\
   &= R +T+ \calR +NL^=,
\end{align*}
with the sets $S_R,S_T$, $S_\calR$ and $S_=$ defined in \eqref{def:SR}-\eqref{def:S=}. 
\subsection{Reaction Term} In this subsection, establish the following estimate on the reaction term 
\begin{align}
    \int_1^t   R d\tau  \lesssim 
    \begin{cases}
        \mu^{-\frac 1 3 } \eps^2_f \eps_q, & \gamma >0,\\
        \vert \ln(\mu)\vert^{1+r} \mu^{-\frac 1 3 }  \eps^2_f \eps_q, & \gamma =0,
    \end{cases}\label{eq:Rest}
\end{align}
which is consistent with  \eqref{NLfest}. To prove the estimate  \eqref{eq:Rest} we split 
\begin{align}\begin{split}
     R&\le \sum_{\substack{k,l\\ k-l,l\neq 0}} \iint d(\eta,\xi) \textbf{1}_{S_R } (\chi^R+\chi^{NR})(k-l,\eta-\xi) \tfrac {\vert \eta l-k\xi\vert \vert k,\eta-kt \vert ^{\tilde \gamma }} {\vert k-l,\eta-\xi-(k-l)t \vert^{1+\gamma  }t ^{\tilde \gamma  }} \\
     &\qquad \qquad \qquad \qquad \cdot \vert A f\vert (k,\eta ) \vert q\vert (k-l,\eta-\xi) \vert \Lambda f\vert (l,\xi)\\
     &\quad +\sum_{\substack{k\\ k\neq 0}} \iint d(\eta,\xi) \textbf{1}_{S_R }  \tfrac {\vert k\vert \vert \xi\vert^{1-\tilde \gamma }  \vert k,\eta-kt \vert ^{\tilde \gamma  }} {\vert k,\eta-\xi-kt \vert^{1+\gamma }} \vert A f\vert (k,\eta ) \vert q\vert (k,\eta-\xi) \vert f\vert (0,\xi)\\
     &\quad +\sum_{\substack{k,l\\ k-l\neq 0} }\iint d(\xi,\eta )  \textbf{1}_{S_R }  \tfrac {\vert \eta l -k\xi \vert} {\vert k-l ,\eta-\xi-(k-l)t \vert^{1+\gamma}} \vert A f_1\vert(k,\eta ) \vert A f_2\vert(l,\xi )\vert  q\vert(k-l,\eta-\xi )\\
      &= R_R+R_{NR}+R_=+R_{low} .\label{eq:R}
\end{split}\end{align}
The resonance cuts $\chi^R$ and $\chi^{NR }$ are defined in \eqref{eq:freqcut}.

\textbf{Bound on $R_{R}$:} On $\chi^R(k-l,\eta-\xi) $ we estimate 
\begin{align*}
    \tfrac {\vert \eta l-k \xi \vert }{\vert k-l \vert}&\le t \vert l,\xi \vert,  \\
    \tfrac {\vert k,\eta-kt\vert  }{\vert k-l \vert}&\le t \vert l,\xi \vert^2.
\end{align*}
Therefore,
\begin{align*}
    \tfrac {\vert \eta l-k\xi\vert \vert k,\eta-kt \vert ^{\tilde \gamma }} {\vert k-l,\eta-\xi-(k-l)t \vert^{1+\gamma  }t ^{\tilde \gamma }}&\lesssim \tfrac t{\langle t-\frac{\eta-\xi}{k-l }\rangle^{1+\gamma } }\vert l,\xi \vert^3.
\end{align*}
Now we distinguish between $\gamma >0$ and $\gamma =0$.  \\
\textbf{Let $\gamma >0 $, } then by Lemma \ref{lem:m1} ii) and iii)
\begin{align*}
    t\tfrac 1 {\langle t-\frac{\eta-\xi}{k-l }\rangle^{1+\gamma } }\vert l,\xi \vert^3&\lesssim \frac t{\min(1, \mu^{\frac 1 3 } t ) }\sqrt{\frac {\p_t m_\gamma }{m_\gamma} (k,\eta)}\sqrt{\frac {\p_t m_\gamma }{m_\gamma} (k-l,\eta-\xi )}\vert l,\xi \vert^8\\
    &\le (t+ \mu^{-\frac 1 3 }) \sqrt{\frac {\p_t {m_\gamma} }{m_\gamma} (k,\eta)}\sqrt{\frac {\p_t m_\gamma }{m_\gamma} (k-l,\eta-\xi )}\vert l,\xi \vert^8
\end{align*}
and so we infer 
\begin{align*}
    \vert R_{R } \vert &\lesssim  (t+ \mu^{-\frac 1 3 })e^{-c\mu^{-\frac 1 3 }t } \Vert \sqrt{\tfrac {\p_t m_\gamma }{m_\gamma}}  A f \Vert_{L^2}\Vert   \sqrt{\tfrac {\p_t m_\gamma }{m_\gamma}} A q  \Vert_{L^2} \Vert A  f \Vert_{L^2}.
    \end{align*}
Integrating in time, using \eqref{boot1} and Lemma \ref{cor:L2} yields 
\begin{align*}
    \int^t_1 \vert R_{R } \vert d\tau&\lesssim \mu^{-\frac 13 }\eps^2_f\eps_q,
\end{align*}
which is consistent with \eqref{eq:Rest}.

\textbf{Let $\gamma =0 $}: We use that either $1+\langle t-\frac{\eta-\xi}{k-l }\rangle\le \mu^{-1}$ or  $1+\langle t-\frac{\eta-\xi}{k-l }\rangle\ge \mu^{-1}$ to estimate with  Lemma \ref{lem:m1} ii) and iii)
\begin{align*}
    t\tfrac 1 {\langle t-\frac{\eta-\xi}{k-l }\rangle }&\lesssim \tfrac t {\langle t-\frac{\eta-\xi}{k-l }\rangle } \tfrac {\vert \ln(\mu^{-1})\vert^{1+r}} {\vert \ln(1+\langle t-\frac{\eta-\xi}{k-l }\rangle)\vert^{1+r}}+ t\tfrac 1{\mu^{-1}-1 } \\
    &\lesssim\frac {t\vert\ln(\mu^{-1})\vert^{1+r}}{\min(1, \mu^{\frac 1 3 } t ) }\sqrt{\frac {\p_t m_0 }{m_0} (k,\eta)}\sqrt{\frac {\p_t m_0 }{m_0} (k-l,\eta-\xi )}\vert l,\xi \vert^7+ t \mu \\
    &\le \vert \ln(\mu)\vert^{1+r} (t+ \mu^{-\frac 1 3 }) \sqrt{\frac {\p_t m_0 }{m_0} (k,\eta)}\sqrt{\frac {\p_t m_0 }{m_0} (k-l,\eta-\xi )}\vert l,\xi \vert^7+ t \mu .
\end{align*}
Therefore, we estimate 
\begin{align*}
     \vert R_{R } \vert &\lesssim \vert \ln(\mu)\vert^{1+r}(t+ \mu^{-\frac 1 3 })e^{-c\mu^{-\frac 1 3 }t } \Vert \sqrt{\tfrac {\p_t m_0 }{m_0}}  A f \Vert_{L^2}\Vert  \sqrt{\tfrac {\p_t m_0 }{m_0}} A q  \Vert_{L^2} \Vert A  f \Vert_{L^2}\\
    &\quad +t \mu e^{-c \mu^{-\frac 13 }t }\Vert A  f \Vert_{L^2}^2 \Vert A  q \Vert_{L^2}. 
\end{align*}
Integrating in time, using \eqref{boot1} and Lemma \ref{cor:L2} yields 
\begin{align*}
    \int_1^t  R_{R }  d\tau&\lesssim \vert \ln(\mu)\vert^{1+r} \mu^{-\frac 13 } \eps_f^2 \eps_q.
\end{align*}
\textbf{Bound on $R_{NR}$: } On $\chi^{NR}(k-l,\eta-\xi) $ we obtain   
\begin{align*}
    \tfrac {1} {\langle t- \frac{\eta-\xi}{k-l}\rangle  }\le \tfrac 1 2  \min(\tfrac 1 {t },\vert\tfrac { \eta -\xi } { k-l }\vert^{-1 }),
\end{align*}
and therefore, 
\begin{align*}
    \tfrac {(\eta l-k\xi)\vert k,\eta-kt \vert^{\tilde \gamma  }} {\vert k-l, \eta-\xi-(k-l)t\vert^{1+\gamma } t^{\tilde \gamma  } }
    &\lesssim  t^{-(\gamma-\tilde \gamma )}\vert l,\xi\vert^2 .
\end{align*}
So we estimate 
\begin{align*}
    \vert R_{NR}\vert &\lesssim \Vert A q_{\neq}\Vert_{L^2}\Vert A  f_{\neq } \Vert_{L^2}\Vert Af \Vert_{L^2} . 
\end{align*}
Integrating in time, using \eqref{boot1} and Lemma \ref{cor:L2} yields 
\begin{align*}
    \int_1^t \vert R_{2}\vert d  \tau &\lesssim \mu^{-\frac 13 } \eps_f^2 \eps_q,
\end{align*}
which is consistent with \eqref{eq:Rest}.

\textbf{Bound on $R_=$:} Here we estimate 
\begin{align*}
    R_=&= \sum_{\substack{k\neq 0}} \iint\textbf{1}_{S_R }  \tfrac {\vert k\vert \vert \xi\vert^{1-\tilde \gamma}  \vert k,\eta-kt \vert ^{\tilde \gamma  }} {\vert k,\eta-\xi-kt \vert^{1+\gamma }}A(k,\eta) \vert A f\vert (k,\eta ) \vert q\vert (k,\eta-\xi) \vert f\vert (0,\xi).
\end{align*}
Since $\gamma \ge \tilde \gamma $ it holds 
\begin{align*}
     \tfrac {\vert k\vert \vert k,\eta-kt \vert ^{\tilde \gamma  }} {\vert k,\eta-\xi-kt \vert^{1+\gamma }} \le \langle t-\tfrac {\eta -\xi}k \rangle^{-1}
\end{align*}
and therefore 
\begin{align*}
    R_=&\lesssim \Vert Af_{\neq} \Vert_{L^2 } \Vert Aq_{\neq} \Vert_{L^2 }\Vert Af \Vert_{L^2 }.
\end{align*}
Integrating in time yields
\begin{align*}    \int_1^t R_=d\tau &\lesssim\mu^{-\frac 13 }\eps^2_f\eps_q.
\end{align*}

\textbf{Bound on $R_{low}$:} We use that on $S_R$
\begin{align*}
    \tfrac {\vert \eta l -k\xi \vert} {\vert k-l ,\eta-\xi-(k-l)t \vert^{1+\gamma}}\lesssim t^{-(1+\gamma) }\vert k-l,\eta-\xi\vert^4 
\end{align*}
and therefore 
\begin{align*}
    R_{low} & \lesssim t^{-(1+\gamma ) }\Vert Af \Vert_{L^2 }^2\Vert Aq_{\neq } \Vert_{L^2 }.
\end{align*}
Integrating in time yields 
\begin{align*}
    \int^t_1 R_{low} d\tau& \lesssim \mu^{-\frac1 6 }\eps_f^2 \eps_q . 
\end{align*}

\subsection{Transport Term} \label{sec:trans} In this subsection we establish the transport estimate 
\begin{align}
    \int^t_1  T d\tau &\lesssim 
    \begin{cases}
        \mu^{-\frac 1 3 } \eps^2_f \eps_q, & \gamma >0,\\
        \vert \ln(\mu)\vert^{1+r} \mu^{-\frac 1 3 }  \eps^2_f \eps_q,& \gamma =0,
    \end{cases} \label{eq:Test}
\end{align}
which is consistent with \eqref{NLfest}. In the remainder of this subsection, we prove \eqref{eq:Test}. Since on $S_T$ we have $8\vert k-l , \eta - \xi \vert \le  \vert l,\xi\vert $ and $k\neq l$, we obtain
\begin{align*}
     T &\le t^{-(1+\gamma ) } \sum_{\substack{k,l\\ k\ne l }} \iint d(\eta,\xi)\textbf{1}_{S_T} \vert l,\xi\vert \left\vert \tfrac { \vert k,\eta-kt\vert ^{\tilde \gamma }} {\vert l,\xi-lt\vert ^{\tilde \gamma  }} A(k,\eta) -A(l,\xi)\right\vert \\
     &\qquad \qquad  \qquad  \qquad \cdot \vert Af\vert (k,\eta) \vert \Lambda^2 q \vert (k-l,\eta-\xi)\vert f\vert  (l,\xi).
\end{align*}
We split the difference of the weight into 
\begin{align*}
    &\tfrac { \vert k,\eta-kt\vert ^{\tilde \gamma }} {\vert l,\xi-lt\vert ^{\tilde \gamma  }} A(k,\eta) -A(l,\xi)\\
    &= (\tfrac {\vert k,\eta-kt \vert^{\tilde \gamma  }} {\vert l,\xi-lt \vert^{\tilde \gamma  }}-1) A(k,\eta )\\
    &\quad + (e^{c \textbf{1}_{k\neq 0}\mu^{\frac 1 3 } t }-e^{c \textbf{1}_{l\neq 0}\mu^{\frac 1 3 } t }) m^{-1}(k,\eta )\langle k,\eta\rangle^Ne^{c \textbf{1}_{l\neq 0}\mu^{\frac 1 3 } t }\\
    &\quad +(\tfrac {\langle k,\eta\rangle ^N}{\langle l,\xi\rangle^N}-1)m^{-1}(k,\eta )\langle l,\xi\rangle^Ne^{c \textbf{1}_{l\neq 0}\mu^{\frac 1 3 } t }\\
    &\quad +(M_L^{-1}(k,\eta)-M_L^{-1}(l,\xi))M_\mu^{-1}(k,\eta ) m_\gamma^{-1}(k,\eta )\langle l,\xi\rangle^Ne^{c \textbf{1}_{l\neq 0}\mu^{\frac 1 3 } t }\\
    &\quad +(M_\mu^{-1}(k,\eta)-M_\mu^{-1}(l,\xi))M_L^{-1}(l,\xi ) m_\gamma^{-1}(k,\eta )\langle l,\xi\rangle^Ne^{c \textbf{1}_{l\neq 0}\mu^{\frac 1 3 } t }\\
    &\quad +(m_\gamma^{-1}(k,\eta)-m_\gamma^{-1}(l,\xi))(M_L^{-1}M_\mu^{-1})(l,\xi ) \langle l,\xi\rangle^Ne^{c \textbf{1}_{l\neq 0}\mu^{\frac 1 3 } t }. 
\end{align*}
By definition of $A$ and the boundedness of the multipliers, we estimate 
\begin{align*}
    &\left\vert \tfrac {\vert k,\eta-kt \vert^{\tilde \gamma}} {\vert l,\xi-lt \vert^{\tilde \gamma }} A(k,\eta) -A(l,\xi)\right\vert\\
    &\lesssim   \big(\vert \tfrac {\vert k,\eta-kt \vert^{\tilde \gamma }} {\vert l,\xi-lt \vert^{\tilde \gamma }}-1\vert e^{c (\textbf{1}_{k\neq 0}-\textbf{1}_{l\neq 0})\mu^{\frac 1 3 } t } +\vert e^{c (\textbf{1}_{k\neq 0}-\textbf{1}_{l\neq 0})\mu^{\frac 1 3 } t } -1\vert  +\left\vert \tfrac {\langle k,\eta\rangle ^N}{\langle l,\xi\rangle^N}-1\right\vert \\
    &\qquad +\vert M_L(k,\eta)-M_L(l,\xi)\vert + \vert M_\mu(k,\eta)-M_\mu(l,\xi)\vert + \vert m_\gamma(k,\eta)-m_\gamma(l,\xi)\vert\big) A(l,\xi).
\end{align*}
Therefore, after rearranging, we split according to 
\begin{align}\begin{split}
     T&\lesssim t^{-(1+\gamma) } \sum_{\substack{k,l\\k\neq l}} \iint \textbf{1}_{S_T }d(\eta,\xi) \Big (\left\vert \tfrac {\vert k,\eta-kt \vert^{\tilde \gamma }} {\vert l,\xi-lt \vert^{\tilde \gamma }}-1\right\vert e^{c (\textbf{1}_{k\neq 0}-\textbf{1}_{l\neq 0})\mu^{\frac 1 3 } t }+\vert e^{c (\textbf{1}_{k\neq 0}-\textbf{1}_{l\neq 0})\mu^{\frac 1 3 } t } -1\vert\\
     &\qquad\quad+\vert \tfrac {\langle k,\eta\rangle ^N}{\langle l,\xi\rangle^N}-1 \vert   +\vert M_L(k,\eta)-M_L(l,\xi)\vert + \vert M_\mu(k,\eta)-M_\mu(l,\xi)\vert +\vert m_\gamma(k,\eta)-m_\gamma(l,\xi)\vert  \Big ) \\
     &\qquad \qquad \qquad\qquad\qquad\qquad  \cdot  \vert l,\xi\vert  \vert Af\vert (k,\eta ) \vert  \Lambda^2 q  \vert (k-l,\eta-\xi) \vert A f\vert  (l,\xi) \\
     &=T_{\tilde \gamma } +T_e +T_N +T_L+ T_\mu +T_m\label{eq:T}
\end{split}\end{align}
and estimate all terms separately. We note that $T_m$ is the most relevant. \\
\textbf{Bound on $T_{\tilde \gamma }$:} For $\gamma=0$ we obtain $\tilde \gamma =0$ and thus $T_{\tilde \gamma }=0$.  We use that for two numbers  $a,b \neq 0$, it holds $\vert a^{\tilde \gamma}-b^{\tilde  \gamma} \vert \lesssim \tfrac 1 {a^{1-{\tilde  \gamma}}+b^{1-{\tilde  \gamma}}}\vert a-b\vert$ (see \cite{bedrossian2013inviscid}, Lemma A.2) to estimate  
\begin{align*}
    \vert  {\vert k,\eta-kt\vert ^{\tilde \gamma  }}- \vert l,\xi-lt\vert ^{\tilde \gamma  }\vert&\lesssim  \tfrac  {\vert k-l,\eta-\xi-(k-l)t\vert}{\vert l,\xi-lt\vert ^{1-\tilde \gamma }+\vert k,\eta-kt\vert ^{1-\tilde \gamma}}\le \tfrac  {\vert k-l,\eta-\xi\vert}{\vert l,\xi-lt\vert ^{1-\tilde \gamma }+\vert k,\eta-kt\vert ^{1-\tilde \gamma }} t .
\end{align*}
We estimate and use Lemma \ref{lem:m1} to infer 
\begin{align*}
    \left\vert \frac {\vert k,\eta-kt \vert^{\tilde \gamma  }} {\vert l,\xi-lt \vert^{\tilde \gamma }}-1\right\vert \tfrac {\vert l,\xi\vert } {t^{1+\gamma }}&\le \frac  {\vert l, \xi \vert }{\vert l,\xi-lt\vert^{\tilde \gamma  } (\vert l,\xi-lt\vert ^{1-\tilde \gamma }+\vert k,\eta-kt\vert ^{1-\tilde \gamma })}t^{-\gamma} \vert k-l,\eta-\xi \vert \\
    &\le\frac  {\vert \xi-lt \vert + \vert l \vert t   }{\vert l,\xi-lt\vert^{\tilde \gamma } (\vert l,\xi-lt\vert ^{1-\tilde \gamma }+\vert k,\eta-kt\vert ^{1-\tilde \gamma })}t^{-\gamma} \vert k-l,\eta-\xi \vert
    \\
    &\le \left(t^{-\gamma} + \textbf{1}_{l\neq 0  } t^{1-\gamma } \frac 1{\langle t-\frac \xi l  \rangle  }\right)\vert k-l,\eta-\xi \vert^2\\
    &\lesssim \left(1+\textbf{1}_{l\neq 0} \left(\frac {t} {\min(1,t\mu^{\frac 13})}\right)^{1-\gamma }\left(\frac{\p_t m_\gamma}{m _\gamma}(k,\eta)\right)^{1/2(1 -\gamma )  }\left(\frac{\p_t m_\gamma}{m_\gamma } (l,\xi)\right)^{1/2(1 -\gamma )}\right)\\
    &\qquad \qquad \cdot \langle k-l,\eta-\xi \rangle^5.
\end{align*}
Therefore,
\begin{align*}
     T_{\tilde \gamma } &\lesssim  \left( t^{1-\gamma}  + \mu^{-\frac13(1-\gamma ) }\right)e^{-c\mu^{\frac 13 } t } \Vert \sqrt{\tfrac{\p_t m_\gamma}{m_\gamma} } Af\Vert_{L^2}^{2-2\gamma }\Vert Af\Vert_{L^2}^{2\gamma }\Vert A  q \Vert_{L^2} \\
    &+\Vert Af_{\neq}\Vert_{L^2}^2\Vert  Aq \Vert_{L^2}+\Vert Af_{\neq}\Vert_{L^2}\Vert Af\Vert_{L^2}\Vert  Aq_{\neq } \Vert_{L^2}.
\end{align*}
Integrating in time, using \eqref{boot1} and Lemma \ref{cor:L2} yields 
\begin{align*}
     \int_1^t T_{\tilde \gamma } d\tau &\lesssim  \mu^{-\frac13}\eps_f^2 \eps_q .
\end{align*}

\textbf{Bound on $T_e$: }
We use $\vert e^x -1\vert \le \vert x\vert e^{\vert x\vert}$ to infer 
\begin{align*}
    \vert e^{(\textbf{1}_{k\neq 0}-\textbf{1}_{l\neq 0})\mu^{\frac 1 3 } t } -1\vert\lesssim \mu^{\frac 1 3 } t e^{c\mu^{\frac 1 3}t }\textbf{1}_{k \neq l }(\textbf{1}_{k =0 }+\textbf{1}_{l=0 }).
\end{align*}
This yields the estimate
\begin{align*}
    T_e& \lesssim \mu^{\frac 1 3}e^{c\mu^{\frac 1 3}t }\Vert \p_vf_=\Vert_{L^2} \Vert Af_{\neq}\Vert_{L^2}\Vert q_{\neq }\Vert_{H^8}\\
    & \lesssim \mu^{\frac 1 3}\Vert \p_vf_=\Vert_{L^2} \Vert Af_{\neq}\Vert_{L^2}\Vert Aq_{\neq }\Vert_{L^2}.
\end{align*}
Integrating in time and using \eqref{boot1} and Lemma \ref{cor:L2} yields 
\begin{align*}
     \int_1^t T_e d\tau &\lesssim  \mu^{-\frac13}\eps_f^2 \eps_q .
\end{align*}

\textbf{Bound on  $T_N$:} By
\begin{align*}
    \left\vert \tfrac {\langle k,\eta\rangle ^N}{\langle l,\xi\rangle^N}-1 \right\vert\lesssim \tfrac {\vert k-l,\eta-\xi \vert }{\vert l,\xi\vert},
\end{align*}
we infer
\begin{align*}
    \vert T_N \vert &\lesssim t^{-1-\gamma }  \Vert Af_{\neq}\Vert_{L^2}^2  \Vert Af\Vert_{L^2}.
\end{align*}
Integrating in time and using \eqref{boot1} and Lemma \ref{cor:L2} yields 
\begin{align*}
     \int_1^t T_N d\tau &\lesssim  \mu^{-\frac16}\eps_f^2 \eps_q .
\end{align*}

\textbf{Bound on  $T_L$:} Here we use that $M_L $ is $\gamma$-admissible in sense of Definition \ref{def:admiss} and therefore 
\begin{align*}
     T_L  &\lesssim (1+ h(t) \mu^{-\frac 1 3 }) \langle  t\mu^{\frac 13 }\rangle^ne^{-c \mu^{\frac 13} t }\Vert A f_{\neq } \Vert_{L^2 } \Vert A f\Vert_{L^2 } \Vert A q_{\neq } \Vert_{L^2 } \\
    &+ (t+\mu^{-\frac 13 } ) \langle t\mu^{\frac 13 }\rangle^ne^{-c \mu^{\frac 13} t }\Vert\sqrt{\tfrac {\p_t m }{m}} A f\Vert_{L^2 }^2 \Vert A q_{\neq } \Vert_{L^2 }\\
    &+ \mu^{\frac 13 } \langle t\mu^{\frac 13 }\rangle^ne^{-c \mu^{\frac 13} t }\Vert\Lambda_t  A f\Vert_{L^2 }\Vert A f\Vert_{L^2 } \Vert A q_{\neq } \Vert_{L^2 }
\end{align*}
for $h(t)\in L^1_t $. Thus, after integrating in time 
\begin{align*}
    \int^t_1  T_L d\tau &\lesssim \mu^{-\frac 13 } \eps_f^2 \eps_q.
\end{align*}

\textbf{Bound on $T_\mu$: } Since $M_\mu$ is admissible, $T_\mu$ is estimated as $T_L$. 

\textbf{Bound on $T_m$: } We split $T_{m}$ to 
\begin{align*}\begin{split}
     T_m&\lesssim  \sum_{\substack{k,l\\k\neq l}} \iint \textbf{1}_{S_T }(\chi^R+\chi^{NR})(l,\xi) \vert m_\gamma (k,\eta)-m_\gamma(l,\xi)\vert\frac {\vert l,\xi\vert } {t^{1+\gamma  }} \\
     &\qquad \qquad\qquad  \cdot \vert Af\vert (k,\eta ) \vert  \Lambda^2 q \vert (k-l,\eta-\xi) \vert A f\vert  (l,\xi) \\
     &=T_{m,R}+T_{m,NR}.
\end{split}\end{align*}
For $T_{m,R}$ we use that Lemma \ref{lem:m1}  vi) 
        \begin{align*}
             \vert m_\gamma(k,\eta)&- m_\gamma(l,\xi  ) \vert
             \lesssim \frac {\langle  k-l,\eta-\xi\rangle} {\langle l\rangle } \\
             &+\frac \xi {\langle l\rangle ^2 } \vert k-l\vert^3 \cdot \begin{cases}
              \left( \tfrac {\p_t  m_\gamma}{  m_\gamma}(l,\xi )\right)^{\tfrac 1 {1+\gamma } }& \gamma>0,  \\
                  \vert \ln(\mu)\vert^{1+r}\tfrac {\p_t  m_0 }{ m_0}(l,\xi) + \mu  & \gamma =0.
             \end{cases}
    \end{align*}
We distinguish between $\gamma >0$ and $\gamma=0$. Let $\gamma >0 $, then we obtain on $\chi^R(l,\xi)$, that 
\begin{align*}
    \vert m_\gamma (k,\eta)-m_\gamma(l,\xi)\vert\frac {\vert l,\xi\vert } {t^{1+\gamma  }}&\lesssim \left( 1+  t^{1-\gamma} \left( \tfrac {\p_t  m_\gamma}{  m_\gamma}(k,\eta )\right)^{\tfrac 1 {2(1+\gamma )} }\left( \tfrac {\p_t  m_\gamma}{  m_\gamma}(l,\xi  )\right)^{\tfrac 1 {2(1+\gamma )} }\right) \langle k-l,  \eta -\xi\rangle^8
\end{align*}
which yields 
\begin{align*}
    T_{m,R}&\lesssim e^{-c\mu^{\frac 13 } t }   \Vert Af\Vert_{L^2}^2 \Vert A  q_{\neq}\Vert_{L^2 }+e^{-c\mu^{\frac 13 } t } t^{1-\gamma }\Vert \sqrt{\tfrac {\p_t m_\gamma }{m_\gamma} }Af\Vert_{L^2}^{\frac {2}{1+\gamma}}\Vert Af\Vert_{L^2}^{\frac {2\gamma}{1+\gamma}} \Vert A  q_{\neq}\Vert_{L^2 }.
\end{align*}
After integrating in time, we obtain 
\begin{align*}
    \int^t_1  T_{m,R}d\tau &\lesssim \mu^{-\frac 13 } \Vert Af\Vert_{L^\infty L^2}^2 \Vert A  q_{\neq}\Vert_{L^\infty L^2 } +\Vert e^{-c\mu^{\frac 13 } t } t^{1-\gamma }\Vert_{L^{\frac {1+\gamma}\gamma}_t}\Vert \sqrt{\tfrac {\p_t m_\gamma }{m_\gamma} }Af\Vert_{L^2 L^2}^{\frac {2}{1+\gamma}}\Vert Af\Vert_{L^\infty L^2}^{\frac {2\gamma}{1+\gamma}}\Vert A q_{\neq } \Vert_{L^\infty L^2 }\\
    &\lesssim\left(\mu^{-\frac 1 3}+\mu^{-\frac 1 3( 1-\gamma + \frac \gamma {1+\gamma })}\right)\eps_f^2 \eps_q \le \mu^{-\frac 1 3}\eps_f^2 \eps_q. 
\end{align*}
For $\gamma=0$ we obtain on $\chi^R(l,\xi)$, that 
\begin{align*}
    \vert m_0 (k,\eta)-m_0(l,\xi)\vert\frac {\vert l,\xi\vert } {t^{1  }}&\lesssim \left( 1 + t\mu +\vert \ln(\mu)\vert^{1+r} t\sqrt{ \tfrac {\p_t  m_0}{  m_0}(k,\eta )}\sqrt{ \tfrac {\p_t  m_0}{  m_0}(l,\xi  )}\right)  \langle k-l,  \eta -\xi\rangle^8.
\end{align*}
Therefore, 
\begin{align*}
    T_{m,R}&\lesssim\vert \ln(\mu)\vert^{1+r} t e^{-c\mu^{\frac 13 } t }\Vert \sqrt{\tfrac {\p_t m_0 }{m_0} }Af\Vert_{L^2}^2\Vert A  q_{\neq}\Vert_{L^2 }+(1+t \mu)  e^{-c\mu^{\frac 13 } t } \Vert A  f\Vert_{L^2}\Vert Af\Vert_{L^2} \Vert A  q_{\neq}\Vert_{L^2 }
\end{align*}
and so integrating in time yields 
\begin{align*}
    \int^t_1  T_{m,R}d\tau &\lesssim \vert\ln(\mu)\vert^{1+r} \mu^{-\frac 13 } \eps_f^2 \eps_q. 
\end{align*}

For $T_{m,NR}$ we use that on $\chi^{NR}(l,\xi)$ 
\begin{align*}
    \vert l,\xi \vert \le 2 \vert l,\xi -lt  \vert. 
\end{align*}
and thus by Lemma \ref{lem:m1} iv) we estimate 
\begin{align*}
    \vert m_\gamma (k,\eta)-m_\gamma(l,\xi)\vert\frac {\vert l,\xi\vert } {t^{1+\gamma  }}\lesssim \mu^{\frac 13 }\vert l,\xi -lt  \vert .
\end{align*}
Therefore,
\begin{align*}
    T_{m,NR}&\lesssim  \mu^{\frac 13 }\Vert  A\Lambda_tf \Vert_{L^2}\Vert q_{\neq}\Vert_{H^8}\Vert Af\Vert_{L^2}.
\end{align*}
After integrating in time, we infer 
\begin{align*}
    \int^t_1  T_{m,NR}d\tau &\lesssim  \mu^{-\frac 13 } \eps_f^2 \eps_q.
\end{align*}

\subsection{Remainder Term} On the set $S_\calR$ we obtain 
\begin{align*}
    \tfrac {\left\vert \vert k,\eta-kt \vert^{\tilde \gamma } A(k,\eta)- \vert l,\xi-lt \vert^{\tilde \gamma} A(l,\xi )\right \vert}{\vert l,\xi-lt \vert^{\tilde \gamma}} \lesssim \langle k,\eta\rangle^{\frac N 2 +2}\langle k-l,\eta-\xi \rangle^{\frac N 2 }(e^{c\mu^{\frac 13 } t\textbf{1}_{k\neq 0 }}+e^{c\mu^{\frac 13 } t\textbf{1}_{l\neq 0} }  )
\end{align*}
therefore 
\begin{align*}
    \calR
    &\lesssim t^{-1-\gamma  } \Vert A f_{\neq} \Vert_{L^2}\Vert A f \Vert_{L^2}\Vert Aq \Vert_{L^2}.
\end{align*}
Integrating in time and using \eqref{boot1} and Lemma \ref{cor:L2} yields 
\begin{align*}
    \int^t_1  \calR d\tau &\lesssim \mu^{-\frac 1 6 } \eps^3. 
\end{align*}

\subsection{Average Nonlinearity}
We bound the term 
\begin{align*}
   NL_f^=&= \vert\langle A  f_1 , [A  \Lambda_t^{\tilde \gamma  }, \tfrac {\p_y }{\vert \p_y\vert^{1+\gamma}} q_=^x]  \p_x \Lambda_t^{-\tilde \gamma }  f_2)\rangle\vert\\
   &\le\sum_{k\neq 0 } \iint d(\eta,\xi)  (\textbf{1}_{S_{R,=}}+\textbf{1}_{S_{T,=}})\tfrac k {\vert \eta -\xi\vert^\gamma} \left\vert \tfrac {\vert k,\eta-kt \vert^{\tilde \gamma }}{\vert k,\xi-kt \vert^{\tilde \gamma}} A(k,\eta)-  A (k,\xi )\right\vert\\
   &\qquad \qquad \qquad \cdot  \left\vert   q \right\vert(0,\eta-\xi)  \vert A  f \vert(k,\eta ) \vert f \vert (l,\xi ) \\
   &= R^=+T^=.
\end{align*}
With $S_{R,=}$ and $S_{R,=}$ defined in \eqref{def:SR=} and \eqref{def:ST=}.\\
\textbf{Bound on $R^=$:} We estimate 
\begin{align*}
    R^=
   &\le\sum_{k\neq 0 } \iint d(\eta,\xi) \textbf{1}_{S_{R,=}} \tfrac k {\vert \eta -\xi\vert^\gamma}\left(\left\vert \tfrac {\vert k,\eta-kt \vert^{\tilde \gamma }}{\vert k,\xi-kt \vert^{\tilde \gamma}} A(k,\eta)\right \vert +\left \vert   A (k,\xi )\right \vert\right) \left\vert   q \right\vert(0,\eta-\xi)  \vert A  f \vert(k,\eta ) \vert f \vert (k,\xi ) 
\end{align*}
On the set $S_{R,=}$ we obtain 
\begin{align*}
     \tfrac {\vert k,\eta-kt \vert^{\tilde \gamma}}{\vert k,\xi-kt \vert^{\tilde \gamma }} &\le 2\langle  \eta-\xi\rangle ^{\tilde \gamma } \vert k,\xi \vert.
\end{align*}
Therefore, 
\begin{align*}
    \left(\left\vert \tfrac {\vert k,\eta-kt \vert^{\tilde \gamma  }}{\vert k,\xi-kt \vert^{\tilde \gamma }} A(k,\eta)\right \vert +\left \vert   A (k,\xi )\right \vert\right)\lesssim e^{c\mu^{\frac 13 }t }  A(0, \eta-\xi)  \langle  \eta-\xi\rangle ^{\tilde \gamma } \vert k,\xi \vert^2.
\end{align*}
Since $\gamma\ge \tilde \gamma $ and since on $S_{R,=}$ we obtain $\vert \eta-\xi \vert \ge 8\vert k,\xi\vert\ge 1  $ we estimate
\begin{align*}
    R^= &\lesssim \Vert  A q\Vert_{L^2 }\Vert  A f_{\neq } \Vert_{L^2 }^2.
\end{align*}
Integrating in time yields 
\begin{align*}
    \int^t_1  R^= d\tau \lesssim \mu^{-\frac 1 3 } \eps^2_f \eps_q .
\end{align*}

\textbf{Bound on $T^=$:} If we have the estimate 
\begin{align}
    k \left\vert \tfrac {\vert k,\eta-kt \vert^{\tilde \gamma }}{\vert k,\xi-kt \vert^{\tilde \gamma }} A(k,\eta)-  A (k,\xi )\right \vert&\lesssim \vert \eta-\xi\vert^\gamma \langle  \eta-\xi\rangle^{N-5}A(k,\xi)  , \label{eq:T=est}
\end{align}
 we estimate  
\begin{align*}
    T^=&\lesssim \Vert A f_{\neq}\Vert_{L^2 }^2 \Vert \langle \p_y\rangle^{N-5}   q_=\Vert_{L^\infty} \lesssim \Vert A f_{\neq}\Vert_{L^2 }^2 \Vert A  q_=\Vert_{L^2}.
\end{align*}
Therefore, integrating in time yields 
\begin{align*}
   \int^t_1   T^=d\tau &\lesssim \mu^{-\frac13 } \eps^2_f \eps_q  .
\end{align*}
For $T^=$ it is only left to prove \eqref{eq:T=est}, we split 
\begin{align*}
    k \left\vert \tfrac {\vert k,\eta-kt \vert^{\tilde \gamma }}{\vert k,\xi-kt \vert^{\tilde \gamma }} A(k,\eta)-  A (k,\xi )\right \vert&\lesssim \tfrac k {\vert k,\xi-kt \vert^{\tilde \gamma}}\left \vert \vert k,\eta-kt \vert^{\tilde \gamma }-\vert k,\xi-kt \vert^{\tilde \gamma} \right \vert  A(k,\eta) \\ 
    &\quad + k \vert A(k,\eta)-A(k,\xi) \vert,\\
    &=I_1+I_2.
\end{align*}
\textbf{Bound on $I_1$:} This term is only nonzero for $\tilde \gamma \neq 0$.  For $\tilde \gamma>0$, we use that for two numbers  $a,b \neq 0$, it holds $\vert a^{\tilde \gamma}-b^{\tilde  \gamma} \vert \lesssim \tfrac 1 {a^{1-{\tilde  \gamma}}+b^{1-{\tilde  \gamma}}}\vert a-b\vert$ (see \cite{bedrossian2013inviscid}, Lemma A.2) to estimate  
\begin{align*}
    \left \vert \vert k,\eta-kt \vert^{\tilde \gamma }-\vert k,\xi-kt \vert^{\tilde \gamma } \right \vert&\lesssim \tfrac {\vert \eta-\xi\vert}{\vert k,\xi-kt \vert^{1-\tilde \gamma } },
\end{align*}
which clearly also holds for $\tilde \gamma=1$. Therefore, 
\begin{align*}
    I_1 &\lesssim \tfrac {\vert \eta-\xi\vert}{\langle t-\frac \xi k \rangle }.
\end{align*}
\textbf{Bound on $I_2$:} Since 
\begin{align*}
    A(k,\eta)= \langle k ,\eta  \rangle^{N} e^{ c \mu^{\frac 1 3}t\textbf{1}_{k\neq 0}}  m^{-1}_\gamma (k,\eta)M_L^{-1}(k,\eta ) M_\mu^{-1}(k,\eta ) 
\end{align*}
the estimate on $I_2$ is a consequence of the  weight $M_L$ and $M_\mu$ beeing $\gamma$-admissible and satisfies Definition \ref{def:admiss} iv), $m_\gamma$ satisfy Lemma \ref{lem:m1} v) and
\begin{align*}
    \vert \langle k ,\eta  \rangle^{N}-\langle k ,\xi  \rangle^{N}\vert \lesssim  \vert \eta -\xi\vert\langle k ,\eta  \rangle^{N-1}. 
\end{align*}

\section{The Boussinesq Threshold}\label{sec:estbou}
In this section, we use Theorem \ref{thm:main} to prove the threshold for Boussinesq equations in Theorem \ref{thm:thres}. We consider the Boussinesq equation \eqref{bouss0} around Couette flow and an affine temperature profile 
\begin{align*}
    V_s (x,y) &= \begin{pmatrix}
        y \\ 0 
    \end{pmatrix}, &
    \Theta_s(x,y)&= \beta^2 y 
\end{align*}
for $\beta \in \R$ and with dissipation $\mu=\kappa=\nu$. The perturbative unknowns 
\begin{align*}
    v(t,x,y) &=V(t,x+yt, y ) - V_s(x,y), \\
    \theta (t,x,y) &=\Theta(t,x+yt, y ) - \Theta_s (x,y),
\end{align*}
satisfy the equation 
\begin{align*}
\begin{split}
    \partial_t v+ e_1 v^y   + (v\cdot  \nabla_t)v  + \nabla_t \pi  &= \mu \Delta_t v -  e_2 \theta , \\
    \partial_t \theta+ (v\cdot \nabla_t) \theta &= \mu \Delta_t\theta  +\beta^2 v^y, \\
    \dive_t (v) &= 0. %\label{eq:NLvth}
\end{split}
\end{align*}
For $\beta>0 $ we define the adapted unknowns 
\begin{align*}
    \zeta_1&= \Lambda_t^{-\frac 12 } \nabla^\perp_t \cdot v, &
    \zeta_2&= \beta^{-1}\Lambda_t^{\frac 12 } \theta  .
\end{align*}
The $\zeta$ unknowns are a modification of the symmetric variables $(Q,Z)$ of \cite{bedrossian21}. Then we obtain the equation
\begin{align}\begin{split}
    \p_t \zeta_1  +\tfrac 1 2 \p_x \p_y^t \Lambda_t^{-2 }   \zeta_1 - \Lambda_t^{-\frac 1 2 }\nabla_t^{\perp}((\overline u\cdot   \nabla) \Lambda_t^{-\frac 3 2 }\nabla_t^{\perp} \zeta_1  )  &= \mu  \Delta_t \zeta_1 + \beta \p_x\Lambda_t^{-1 } \zeta_2,\\
     \p_t \zeta_2-\tfrac 1 2\p_x \p_y^t \Lambda_t^{-2 }   \zeta_2   +\qquad \ \, \Lambda_t^{\frac 1 2 }  ( ( \overline u  \cdot \nabla) \Lambda_t^{-\frac 1 2 }\quad  \zeta_2)    &= \mu  \Delta_t \zeta_2  +\beta \p_x \Lambda_t^{-1} \zeta_1,\\
     \overline u =\Lambda_t^{-\frac 32 } \nabla^\perp \zeta_1 . \label{eq:main}
\end{split}\end{align}

For these adapted unknowns, we contain the following stability theorem: 
\begin{theorem}[Boussinesq part of Theorem \ref{thm:thres}]\label{thm:Bouss}
Let $\beta>\tfrac 12  $, $N\ge 12$. Then there exists $ \delta_\beta ,c >0 $ such that for $0<\mu \le \frac12$ and initial data which satisfies 
    \begin{align*}
        \Vert \zeta_{in } \Vert_{H^{N } }=\eps &\le \delta  \mu^{\frac 1 3},
    \end{align*}
 the global in time solution $\zeta $ to \eqref{eq:main} is in $ C_t H^N$ and satisfies the following:\\
     \textbf{Stability of adapted unknowns} 
        \begin{align*}
        \sup_{t\ge 0}\Vert \zeta \Vert_{H^{N} }&\lesssim \eps .
    \end{align*}    
     \textbf{Inviscid damping and Enhanced dissipation estimates,} for all $t \ge 0$ we obtain 
    \begin{align*}
        \Vert \zeta_{\neq}  \Vert_{H^{N }}+ \langle t\rangle^{\frac 12 }\Vert u^x_{\neq},\theta_{\neq}  \Vert_{H^{N-\frac 1 2  }}+  \langle t\rangle^{\frac 32 }\Vert u^y \Vert_{H^{N-\frac 3 2  }}&\lesssim \eps e^{-c\mu^{\frac 1 3 } t }.
    \end{align*}  
\end{theorem}

For the remainder of this section, we establish Theorem \ref{thm:Bouss}. In this section, we use the weights
\begin{align*}
    A(k,\eta )&= \langle k ,\eta  \rangle^{N} e^{ c \mu^{\frac 1 3}t\textbf{1}_{k\neq 0}}m^{-1}(k,\eta),\\
    m(k,\eta )&=  m_{\frac 12 }(k,\eta ) M_L^\theta(k,\eta ) M_\mu(k,\eta ). 
\end{align*}
with the \emph{linear weight} defined as 
\begin{align}\begin{split}\label{eq:M1weight}
    M_L^\theta(t,k,\eta) &= \exp\left(2c^{-1}\int_{-\infty}^t \textbf{1}_ {\vert \tau- \frac \eta k \vert\le 2\mu^{-\frac 1 6 }c^{-1}} \tfrac 1 {\langle\frac{\eta}{k}- \tau\rangle^{3 }} d\tau\right),\qquad\qquad  k\neq 0,\\
    M_L^\theta(t,0,\eta)&=M_L^\theta(t,1,\eta), 
\end{split}\end{align}
where  $c:=\tfrac {1}4 (1-\frac 1{4\beta})$. We define the energy 
\begin{align*}
    \calE(t)&:= \tfrac 1 2  \Vert A\zeta \Vert_{L^2}^2 +\frac 1 {2\beta } \langle \p_y^t \Lambda_t^{-1} A\zeta_{1}, A\zeta_{2}  \rangle. 
\end{align*}
For times $0\le t\le 1 $ the norm stays bounded by classical energy estimates,  $\calE (1)\le C_1 \eps^2$.  In the following, we prove that the energy $\calE$ stays bounded for times $t\ge 1$. We do this by a Bootstrap method, for $1<t$ the bootstrap assumption states \\
\textbf{Bootstrap assumption:}
\begin{align}
\begin{split}\label{bootbou}
   \sup_{1\le t\le t^\ast} \calE (t)  +4c\int_1^t \mu \Vert \nabla_t A  \zeta\Vert^2_{L^2} +\Vert \sqrt {\tfrac {\p_t m }m} A  \zeta \Vert_{L^2}^2 d\tau&\le C_2\eps^2.
\end{split}
\end{align}

\begin{prop}\label{pro:Boot}
Let $\zeta$ be a solution to \eqref{eq:main} which satisfies the Bootstrap hypotheisis \eqref{bootbou} for $1\le t$, then there exists a for a $C_3>0 $ such that 
\begin{align}
   \label{IE:MN}  \calE (t)  +4c\int_1^t \mu \Vert \nabla_t A  \zeta\Vert^2_{L^2} +\Vert \sqrt {\tfrac {\p_t m }m} A  \zeta \Vert_{L^2}^2 d\tau&\le(C_1+\frac 12C_2+  C_3\delta )\eps ^2.
\end{align}
\end{prop}
With that proposition we prove Theorem \ref{thm:Bouss}
\begin{proof}[Proof of Theorem \ref{thm:Bouss}]
    First, we show that \eqref{bootbou} holds for all times $t>1$. For the sake of contradiction, we assume that $t^\ast$ is the maximal time such that \eqref{bootbou} holds.   With Proposition \ref{pro:Boot} and choosing $\delta<\frac 14 \min(1,C_2/C_3)$, \eqref{bootbou} holds with $<$. By local estimates, this contradicts the maximality of $t^\ast$ and thus we infer a global bound. Therefore, using $\Vert A\zeta_{\neq}  \Vert_{L^2}\approx e^{-c\mu^{\frac 13 }t  } \Vert \zeta_{\neq}  \Vert_{H^N}$ we infer the stability estimates. The inviscid damping and estimates are direct consequence of $\Vert \Lambda^{-\gamma}_L\zeta_{\neq}  \Vert_{H^{N-\gamma }}\lesssim  t^{-\gamma }  \Vert \zeta_{\neq}  \Vert_{H^{N }}$. Therefore, Theorem \ref{thm:Bouss} holds. 
 \end{proof}
The remainder of this section will focus on proving Proposition \ref{pro:Boot} by using two lemmas. The following lemma computes identities for the main energy $\calE$ to allow the usage of Theorem \ref{thm:main}.
\begin{lemma}\label{lem:split}
Let  the Bootstrap hypothesis \eqref{bootbou} hold for $0\le t \le t^\ast $, then we obtain the identity 
\begin{align*}
    \p_t \calE (t)+ 4c \mu   \Vert \nabla_t A  \zeta\Vert^2_{L^2} +4c \Vert  \sqrt {\tfrac {\p_t m }m} A  \zeta \Vert_{L^2}^2 &\le L_1+ \sum_{f\in \Gamma }NL_f+ LNL.
\end{align*}
with $\Gamma = \{(f_1,f_2):f_1 , f_2 \in  \{\p_y^t \Lambda_t^{-1}\zeta_1,\p_x \Lambda_t^{-1}\zeta_1,\zeta_2\}\}$.
Where we denoted:\\
\textbf{Linear term:}
\begin{align*}
    L_1 &= 2c\mu^{\frac 1 3 }\Vert A\zeta\Vert^2_{L^2} + \vert\langle A \p_x^3 \Lambda_t^{-3 }\zeta_1 , A \zeta_2 \rangle\vert .
\end{align*}
\textbf{Main nonlinearity:}
\begin{align*}
   NL_f&= \vert\langle A  f_1 , [A  \Lambda_t^{\frac 1 2 }, \nabla^\perp \Lambda^{-\frac 32 }_t\zeta_1 ]\cdot   \nabla \Lambda_t^{-\frac 1 2 }  f_2)\rangle\vert.
\end{align*}
\textbf{Lower nonlinearities:}
\begin{align*}
    LNL &= \vert \langle A (   (\overline u \cdot  \nabla) \Lambda_t^{-\frac 3 2 }\nabla^\perp \zeta_1  ), A \p_x\nabla \Lambda_t^{-\frac 3 2 }\zeta_2  \rangle\vert.
\end{align*}

\end{lemma}
We also need that $M_L^\theta$ is admissible and bounds the linear terms. 

\begin{lemma}[Properties of $M_L^\theta $]\label{lem:M1} Let $(k,\eta), (l,\xi) \in \bbZ\times \R $, then $M^\theta_L$ is monoton increasing and satisfies 
\begin{enumerate}[label=(\roman*)]
    \item For $k\neq 0 $ it holds that 
\begin{align*}
    {\langle t-\tfrac{\eta}{k}\rangle^{-\frac 32 }}&\le c {\tfrac {\p_t M_L^\theta  }{M_L^\theta  }}(k,\eta ) +c  \mu^{\frac 1 3 }. 
\end{align*}
    \item It holds the estimate  $1 \le M^\theta_L\le e^{2c^{-1}}. $ 
    \item For $\vert l,\xi \vert \ge \vert k-l,\eta-\xi\vert $ it holds
\begin{align}\begin{split}
    \vert M_L^\theta (k,\xi) - M_L^\theta (l,\xi)\vert  \frac {\vert l,\xi\vert}{t^{\frac 32 }} &\lesssim \Big( \langle \mu^{-\frac 13 } t^{-\frac 32 } \rangle + ( t+\mu^{-\frac 1 3 })\sqrt{\tfrac {\p_t m_{\frac 12 }}{ m_{\frac 12 }}(k,\eta )}\sqrt{\tfrac {\p_t m_{\frac 12 }}{ m_{\frac 12 }}(l,\xi )}\\
    &\qquad \qquad +t^{-\frac 32 } \mu^{\frac 12 }\vert l,\xi-lt \vert\Big)\langle k-l,\eta-\xi\rangle^6.\label{eq:M1diff1}
\end{split}\end{align}
         \item For $\vert k,\xi \vert \ge \vert\eta-\xi\vert $ it holds
         \begin{align*}
    \vert M_L^\theta (k,\eta)-M_L^\theta (k,\xi)\vert&\lesssim \tfrac {\vert \eta-\xi\vert}{ \langle k\rangle}. 
\end{align*}
     \end{enumerate}
In particular, for $M^\theta _L$ is $\gamma =\frac 12 $-admissibles in sense of Definition \ref{def:admiss}.
\end{lemma}
Now we prove Proposition \ref{pro:Boot} and by using Theorem \ref{thm:main}, Lemma \ref{lem:split} and \ref{lem:M1}. Then afterwards we show  Lemma \ref{lem:split} and \ref{lem:M1}.\\
\textbf{Proof of Proposition \ref{pro:Boot}:} By Lemma \ref{lem:split}, we only need to bound 
\begin{align}
    \int_1^t L d\tau &\le \frac 1 2 C_2 \eps^2 , \label{eq:BL}\\
    \int_1^t NL_f d\tau &\lesssim \delta \eps^2, \label{eq:NL}\\
    \int_1^t LNL d\tau &\lesssim \delta \eps^2. \label{eq:LNL}
\end{align}
By Lemma \ref{lem:M1},  $M_L^\theta$ is admissible to apply Theorem \ref{thm:main} with $\gamma=\tilde \gamma =\tfrac 1 2$ and $f\in \Gamma$ and $q= \zeta_1 $ and therefore, since $\mu^{-\frac 13 } \eps \le \delta $, we obtain estimate \eqref{eq:NL}. 

\textbf{Linear term }We bound the linear term 
\begin{align*}
    L &= 2c\mu^{\frac 1 3 }\Vert A\zeta_{\neq }\Vert^2_{L^2} + \vert\langle A \vert \p_x\vert^3 \Lambda_t^{-3 }\zeta_{1,\neq} , A \zeta_{2,\neq} \rangle \vert  .
\end{align*}
Using Lemma \ref{lem:M1} we obtain 
\begin{align*}
     \tfrac 1 {2\beta }\vert \langle \p_x^3\Lambda_t^{-3} A \zeta_{1,\neq},A \zeta_{2,\neq} \rangle \vert   &\le \tfrac 1 {2\beta }  \sum_k \int d\eta \tfrac 1 {(1+(t-\frac \eta k )^2)^{\frac 3 2 }} \vert A\zeta\vert^2(k,\eta)  \\
    &\le \tfrac 12  c\Vert \sqrt{\tfrac {\p_t M_L^\theta }{M_L^\theta}} \zeta_{\neq}  \Vert^2_{L^2}+ \tfrac 12  c\mu^{\frac 1 2 } \Vert A \zeta_{\neq } \Vert_{L^2}^2 .
\end{align*}
Furthermore, with Lemma \ref{lem:M2}
\begin{align*}
    c \mu^{\frac 1 3 }\Vert A \zeta_{\neq }\Vert\le c \Vert A \sqrt{\tfrac {\p_t M_\mu }{M_\mu}} \zeta_{\neq }\Vert_{L^2}^2 + c \mu  \Vert A \nabla_t   \zeta_{\neq }\Vert_{L^2}^2 
\end{align*}
and thus 
\begin{align*}
    L&\le 2c\left(\Vert \sqrt{\tfrac {\p_t M }M} A\zeta_{\neq} \Vert_{L^2}^2  + \mu\Vert \nabla_t \zeta_{\neq}\Vert_{L^2}^2 \right).
\end{align*}
Integrating in time yields \eqref{eq:BL}.

\textbf{Lower Nonlinear Term:} We bound the term 
\begin{align*}
    LNL &= \vert \langle A(   (\overline u  \nabla) \Lambda_t^{-\frac 3 2 }\nabla^\perp \zeta_1  ), A\p_x\nabla \Lambda_t^{-\frac 3 2 }\zeta_2  \rangle\vert \\
    &\le\sum_{\substack{k,l\\k\neq l }}\iint d(\xi,\eta ) \frac {l  \vert \eta l-k\xi\vert^2 }{\vert k-l,\eta-\xi-(k-l)t \vert^{\frac 3 2 }\vert k,\eta-kt \vert^{\frac 3 2 }\vert l,\xi-lt \vert^{\frac 3 2 }}\\
    &\qquad \qquad \cdot A(k,\eta) \vert A\zeta_2 \vert(k,\eta)\vert \zeta_1 \vert(l,\xi)\vert  \zeta_1\vert(k-l,\eta-\xi ).
\end{align*}
With the estimate 
\begin{align}
    \frac {l  \vert \eta l-k\xi\vert^2 }{\vert k-l,\eta-\xi-(k-l)t \vert^{\frac 3 2 }\vert k,\eta-kt \vert^{\frac 3 2 }\vert l,\xi-lt \vert^{\frac 3 2 }}\le 1 \label{eq:ONLest}
\end{align}
and the fact that always two of $k$, $l$ and $k-l$ are nonzero we obtain 
\begin{align*}
    LNL\lesssim \Vert A \zeta_{\neq} \Vert_{L^2}^2 \Vert A \zeta \Vert_{L^2}.
\end{align*}
Therefore, after integrating in time 
\begin{align*}
    \int LNL\ d\tau \lesssim \mu^{-\frac 13 } \eps^3 \lesssim \delta \eps^2 .
\end{align*}
It is left to prove \eqref{eq:ONLest}.  Let $k,l,k-l\neq 0$, then  we use the splittings
\begin{align*}
    \eta l-k\xi&= (\eta-kt ) l - k (\xi-lt)\\
    &=(\eta-kt ) (k-l) - k (\eta-\xi-(k-l)t)\\
    &=(\eta-\xi-(k-l)t ) l -(\xi-lt) (k-l) 
\end{align*}
to estimate 
\begin{align*}
    (\eta l-k\xi)^2&\le \vert (\eta-kt ) l((\eta-\xi-(k-l)t ) l -(\xi-lt) (k-l) )\vert \\
    &+\vert  k (\xi-lt)((\eta-kt ) (k-l) - k (\eta-\xi-(k-l)t))\vert \\
    &\le\vert (\eta-kt ) (\eta-\xi-(k-l)t ) \vert l^2\\
    &+\vert \eta-kt \vert \vert\xi-lt\vert  \vert k, l\vert  \vert k-l\vert \\
    &+ \vert  \xi-lt\vert \vert \eta-\xi-(k-l)t)\vert k^2.
\end{align*}
Therefore,
\begin{align*}
    &\frac {\vert l\vert  \vert \eta l-k\xi\vert^2 }{\vert k-l,\eta-\xi-(k-l)t \vert^{\frac 3 2 }\vert k,\eta-kt \vert^{\frac 3 2 }\vert l,\xi-lt \vert^{\frac 3 2 }}\\
    &\qquad \le\frac {\vert k,l\vert^2  }{\vert k-l,\eta-\xi-(k-l)t \vert^{\frac 1 2 }\vert k,\eta-kt \vert^{\frac 1 2 }\vert l,\xi-lt \vert^{\frac 1 2 }}\\
    &\qquad\le  \frac{1}{\langle t-\frac{\eta}{k} \rangle^{\frac 12} \langle t-\frac{\eta-\xi}{k-l}\rangle^{\frac 12 } \langle t-\frac{\xi}{l}\rangle^{\frac 12 } }\le 1. 
\end{align*}
Let $k=0$ or $k=l$, WLOG $k=l $ then 
\begin{align*}
    \frac {k^3 \vert \eta -\xi\vert^2 }{\vert \eta-\xi \vert^{\frac 3 2 }\vert k,\eta-kt \vert^{\frac 3 2 }\vert k,\xi-kt \vert^{\frac 3 2 }}\lesssim 1.
\end{align*}
This proves \eqref{eq:ONLest}, therefore we obtain \eqref{eq:LNL} which yields Proposition \ref{pro:Boot}.

\textbf{Proof of Lemma \ref{lem:split}}\label{sec:comlem}
The time derivative of $\calE$ satisfies
\begin{align*}
    \p_t \calE &=  \langle  A\zeta, \p_t A \zeta\rangle+ \langle  A\zeta, A\p_t \zeta\rangle  \\
    &\quad-\frac 1 {2\beta } \langle   A\p_t(\p_y^t \Lambda_t^{-1}) \zeta_1+\p_y^t \Lambda_t^{-1}\p_t  \zeta_1, A\zeta_2  \rangle -\frac 1 {2\beta } \langle A \p_y^t \Lambda_t^{-1} \zeta_1, A\p_t\zeta_2  \rangle \\
    &= c\mu^{\frac 1 3 }\Vert A\zeta \Vert_{L^2}^2 -  \Vert \sqrt{\tfrac {\p_t m } m } A \zeta \Vert_{L^2}^2  \\
    &+ \langle A\zeta_1 , A\left( -\tfrac 1 2 \p_x \p_y^t \Lambda_t^{-2 }   \zeta_1 + \Lambda_t^{-\frac 1 2 }\nabla_t((\overline u  \cdot \nabla) \Lambda_t^{-\frac 3 2 }\nabla_t \zeta_1  )  +\mu  \Delta_t \zeta_1 + \beta \p_x\Lambda_t^{-1 } \zeta_2\right)\rangle \\
    &+ \langle A\zeta_2 ,A\left( \tfrac 1 2\p_x \p_y^t \Lambda_t^{-2 }   \zeta_2   -  \Lambda_t^{\frac 1 2 }  ( (\overline u \cdot  \nabla) \Lambda_t^{-\frac 1 2 }  \zeta_2)    + \mu  \Delta_t \zeta_2  +\beta \p_x \Lambda_t^{-1} \zeta_1\right)\rangle \\
    &- \frac 1 {2\beta } \langle \vert \p_x\vert^3 \Lambda_t^{-3} A\zeta_1,A\zeta_2 \rangle+ \tfrac 1{\beta } \langle (c\mu^{\frac 1 3 } - \tfrac {\p_t m} m)  \p_y^t \Lambda_t^{-1} \zeta_{1}, A\zeta_{2}  \rangle  \\
    &-\frac 1 {2\beta } \langle A\big ( \p_y^t \Lambda_t^{-1}(-\tfrac 1 2 \p_x \p_y^t \Lambda_t^{-2 }   \zeta_1 -\Lambda_t^{-\frac 1 2 }\nabla_t((\overline u \cdot  \nabla) \Lambda_t^{-\frac 3 2 }\nabla_t \zeta_1  ) + \mu  \Delta_t \zeta_1  + \beta \p_x\Lambda_t^{-1 } \zeta_2)\big )  , A\zeta_{2}  \rangle\\
    &-\frac 1 {2\beta } \langle \p_y^t \Lambda_t^{-1} \zeta_{1}, A(\tfrac 1 2\p_x \p_y^t \Lambda_t^{-2 }   \zeta_2   - \Lambda_t^{\frac 1 2 }  ( (\overline u \cdot \nabla) \Lambda_t^{-\frac 1 2 }  \zeta_2)    + \mu  \Delta_t \zeta_2  +\beta \p_x \Lambda_t^{-1} \zeta_1)  \rangle. 
\end{align*}
First, we reorganize the linear terms, note that the linear interaction appears only for the non $x$-average part, therefore we omit writing the $\neq$. We obtain the cancellations
\begin{align*}
    & \langle A\zeta_1 , A( -\tfrac 1 2 \p_x \p_y^t \Lambda_t^{-2 }   \zeta_1 +\beta \p_x \Lambda_t^{-1} \zeta_1)\rangle + \langle A\zeta_2 ,A( \tfrac 1 2\p_x \p_y^t \Lambda_t^{-2 }   \zeta_2  +\beta \p_x \Lambda_t^{-1} \zeta_1 )\rangle \\
    &-\frac 1 {2\beta } \langle A \p_y^t \Lambda_t^{-1}(-\tfrac 1 2 \p_x \p_y^t \Lambda_t^{-2 }   \zeta_1  + \beta \p_x\Lambda_t^{-1 } \zeta_2), A\zeta_2  \rangle\\
    &-\frac 1 {2\beta } \langle \p_y^t \Lambda_t^{-1} \zeta_1, A(\tfrac 1 2\p_x \p_y^t \Lambda_t^{-2 }   \zeta_2    +\beta \p_x \Lambda_t^{-1} \zeta_1)  \rangle=0. 
\end{align*}
For the other terms, we estimate 
\begin{align*}
    -\tfrac 1{\beta } \langle (c\mu^{\frac 1 3 } &+ \tfrac {\p_t m} m)  \p_y^t \Lambda_t^{-1} \zeta_1, A\zeta_2  \rangle- 2 \Vert \sqrt{\tfrac {\p_t m } m } A \zeta \Vert_{L^2}\le \tfrac c{2\beta }\mu^{\frac 1 3 }\Vert A\zeta \Vert -2(1-\tfrac 1 {2\beta })\Vert \sqrt{\tfrac {\p_t m } m } A \zeta \Vert_{L^2}.
\end{align*}

Therefore, since $c= \tfrac 1 4 (1-\frac 1{2\beta } ) $ we infer 
\begin{align}\begin{split}
    \p_t \calE +  4c\Vert \nabla_t  A \zeta \Vert_{L^2}&+  4c\Vert \sqrt{\tfrac {\p_t m } m } A \zeta \Vert_{L^2}\le  L \\
    &+  \langle A\zeta_1 , A(   \Lambda_t^{-\frac 1 2 }\nabla_t((\overline u \cdot  \nabla) \Lambda_t^{-\frac 3 2 }\nabla_t \zeta_1  ))\rangle\\
    &- \langle A\zeta_2 ,A(  \Lambda_t^{\frac 1 2 }  ( (\overline u \cdot  \nabla) \Lambda_t^{-\frac 1 2 }  \zeta_2)    )\rangle\\
    &- \frac 1 {2\beta } \langle A( \p_y^t \Lambda_t^{-\frac 3 2 }\nabla_t((\overline u \cdot  \nabla) \Lambda_t^{-\frac 3 2 }\nabla_t \zeta_1  )), A\zeta_{2}  \rangle\\
    &+\frac 1 {2\beta } \langle \p_y^t \Lambda_t^{-1} \zeta_{1}, A(\Lambda_t^{\frac 1 2 }  ( (\overline u \cdot  \nabla) \Lambda_t^{-\frac 1 2 }  \zeta_{2})    )  \rangle. \label{eq:Eest}
\end{split}\end{align}
For the nonlinear terms we show that they are of the form 
\begin{align*}
\langle A f_1 , [A \Lambda_t^{\frac 1 2 },\overline u] \cdot  \nabla \Lambda_t^{-\frac 1 2 }  f_2)\rangle + \text{lower order terms }
\end{align*}
for $f\in \Gamma$. By partial integration we obtain for  $f\in \{ \underbrace{\Lambda_t^{-1}\nabla^\perp_t \zeta_1}_{=:z}, \zeta_2\} $, that 
\begin{align}\begin{split}
    \langle A\zeta_1 , A(   \Lambda_t^{-\frac 1 2 }\nabla_t((\overline u \cdot  \nabla) \Lambda_t^{-\frac 3 2 }\nabla_t \zeta_1  ))\rangle
    &=- \langle Az , \Lambda^{\frac 1 2 }_t A(   (\overline u  \cdot \nabla) \Lambda_t^{-\frac 1 2 } z )\rangle\\
    &=-  \langle Az , [\Lambda^{\frac 1 2 } _tA, \overline u] \cdot  \nabla \Lambda_t^{-\frac 1 2 } z  ))\rangle\\
    \langle A\zeta_2 ,A(  \Lambda_t^{\frac 1 2 }  ( (\overline u  \cdot \nabla) \Lambda_t^{-\frac 1 2 }  \zeta_2)    )\rangle&=\langle A\zeta_2 ,[  \Lambda_t^{\frac 1 2 }A,\overline u] \cdot  \nabla \Lambda_t^{-\frac 1 2 }  \zeta_2)    )\rangle \label{eq:eqcomm1}
\end{split}\end{align}
and thus we see that the first two nonlinear terms are in the desired form. 

Now we symmetrize the  third and fourth nonlinear terms of \eqref{eq:Eest}, which we denote as
\begin{align*}
    \langle A(   \p_y^t \Lambda_t^{-\frac 3 2 }\nabla_t((\overline u \cdot  \nabla) \Lambda_t^{-\frac 3 2 }\nabla_t \zeta_1  )), A\zeta_2  \rangle- \langle \p_y^t \Lambda_t^{-1} \zeta_1, A(\Lambda_t^{\frac 1 2 }  ( (\overline u  \cdot \nabla) \Lambda_t^{-\frac 1 2 }  \zeta_2)    )  \rangle
    :=\tilde I_1 +I_2. 
\end{align*}
We use $(\p_y^t)^2=-\Lambda_t^2 -\p_x^2$ to rewrite 
\begin{align*}
    \tilde I_1= \langle A(   \p_y^t \Lambda_t^{-\frac 3 2 }\nabla_t((\overline u \cdot  \nabla) \Lambda_t^{-\frac 3 2 }\nabla_t \zeta_1  )), A\zeta_{2}  \rangle&= \langle A(   (\p_y^t)^2\Lambda_t^{-\frac 3 2 }((\overline u  \cdot \nabla) \Lambda_t^{-\frac 3 2 }\p_y^t \zeta_1  )), A\zeta_2  \rangle\\
    &\quad + \langle A(   \p_x \p_y^t \Lambda_t^{-\frac 3 2 }((\overline u \cdot  \nabla) \Lambda_t^{-\frac 3 2 }\p_x \zeta_1  )), A\zeta_2  \rangle\\
    &= -\langle A(  \Lambda_t^{\frac 1 2 }((\overline u \cdot  \nabla) \Lambda_t^{-\frac 3 2 }\p_y^t \zeta_1  )), A\zeta_2  \rangle\\
    &\quad - \langle A(   \p_x^2 \Lambda_t^{-\frac 3 2 }((\overline u \cdot  \nabla) \Lambda_t^{-\frac 3 2 }\p_y^t \zeta_1  )), A\zeta_2  \rangle\\
    &\quad + \langle A(   \p_x \p_y^t \Lambda_t^{-\frac 3 2 }((\overline u \cdot  \nabla) \Lambda_t^{-\frac 3 2 }\p_x \zeta_1  )), A\zeta_2  \rangle\\
    :&=I_1+(\tilde I_1-I_1).
\end{align*}
With the identity 
\begin{align*}
    0 &= \langle \langle A\zeta_2  , (\overline u\cdot \nabla) A\p_y^t \Lambda_t^{-1} \zeta_1 \rangle+  \langle A\p_y^t \Lambda_t^{-1} \zeta_1  , (\overline u\cdot \nabla)  A\zeta_2 \rangle
\end{align*}
we infer 
\begin{align*}
     -(I_1+I_2) &=\langle A  \Lambda_t^{\frac 1 2 }((\overline u \cdot  \nabla) \Lambda_t^{-\frac 3 2 }\p_y^t \zeta_1  ), A\zeta_2  \rangle+\langle \p_y^t \Lambda_t^{-1} \zeta_1, A\Lambda_t^{\frac 1 2 }  ( (\overline u \cdot  \nabla) \Lambda_t^{-\frac 1 2 }  \zeta_2)      \rangle\\
     &= \langle A\zeta_2 , [A\Lambda_t^{\frac 12 }, \overline u ]\cdot \nabla \p_y^t \Lambda_t^{-\frac 3 2 } \zeta_1 \rangle + \langle \p_y^t \Lambda_t^{-1} \zeta_1, [A\Lambda_t^{\frac 1 2 } , \overline u ] \cdot \nabla \Lambda_t^{-\frac 12 }  \zeta_2 \rangle .
\end{align*}
For  $z^x=-\Lambda_t^{-1 }\p_y^t \zeta_1$ we have 
\begin{align}
     \vert I_1+I_2\vert &\le NL_{(z^x,\zeta_2)}+NL_{(\zeta_2,z^x)}.\label{eq:commNL}
\end{align}
Furthermore, we obtain 
\begin{align*}
    \tilde I_1-I_1&=-\langle A(   \p_x^2 \Lambda_t^{-\frac 3 2 }((\overline u \cdot  \nabla) \Lambda_t^{-\frac 3 2 }\p_y^t \zeta_1  )), A\zeta_2  \rangle+ \langle A   \p_x \p_y^t \Lambda_t^{-\frac 3 2 }((\overline u \cdot  \nabla) \Lambda_t^{-\frac 3 2 }\p_x \zeta_1  ), A\zeta_2  \rangle\\
    &= \langle A(   (\overline u \cdot  \nabla) \Lambda_t^{-\frac 3 2 }\nabla_t^\perp \zeta_1  ), A\p_x\nabla_t \Lambda_t^{-\frac 3 2 }\zeta_2  \rangle\\
    &= \langle A(   (\overline u \cdot  \nabla) \Lambda_t^{-\frac 3 2 }\nabla^\perp \zeta_1  ), A\p_x\nabla \Lambda_t^{-\frac 3 2 }\zeta_2  \rangle
\end{align*}
which yields
\begin{align}\label{eq:commONL}
    \vert \tilde I_1-I_1\vert & = LNL. 
\end{align}
Equation \eqref{eq:Eest} and the identities \eqref{eq:eqcomm1}, \eqref{eq:commONL} and \eqref{eq:commNL} yield Lemma \ref{lem:split}.

\section{The MHD Threshold}\label{sec:estmhd}
In this section, we use Theorem \ref{thm:main} to prove the threshold for the MHD equations in Theorem \ref{thm:thres}. We consider the MHD equations \eqref{MHD0} around Couette flow and a constant magnetic field 
\begin{align*}
    V_s (x,y) &= \begin{pmatrix}
        y \\ 0 
    \end{pmatrix}, &
    B_s(x,y)&= \alpha 
\end{align*}
for $\alpha \in \R^2\setminus\{ 0\} $. The perturbative unknowns 
\begin{align*}
    v(x,y,t)&= V(x+yt,y,t )- V_s, \\
    b(x,y,t)&= B(x+yt,y,t )- B_s,
\end{align*}
where the change of variables $x\mapsto x+yt$ follows the characteristics of the Couette flow. For these unknowns, equation \eqref{MHD0} becomes 
\begin{align}
\begin{split}
    \partial_t v + v^y e_1 - 2\partial_x \Delta^{-1}_t  \nabla_t v^y  &=  \nu  \Delta_t v+ (\alpha \cdot \nabla_t) b  + (b\cdot \nabla_t) b- (v\cdot \nabla_t) v-\nabla_t \pi , \\
    \partial_t b - b^y e_1 \qquad \qquad \quad \quad  \ &= \kappa \Delta_t   b+ (\alpha \cdot \nabla_t) v  +(b\cdot \nabla_t) v -(v\cdot\nabla_t) b,\\
    \dive_t(v)=\dive_t(b) &= 0.\label{eq:vb}
\end{split}
\end{align}
\begin{theorem}[MHD part of Theorem \ref{thm:thres}]\label{thm:MHD}
    Let $N\ge 12$, $\alpha  \in \R^2\setminus \{ 0 \}$ and $0<\kappa , \nu\le \frac 1 {10} $, with either $\alpha_2 \neq 0$ or $\kappa^3 \le \nu \le \kappa^{\frac 13 }$, $\mu: =\min(\nu,\kappa) $ and $r>0$. Then there exist $\delta_{\alpha,r},c >0$ such that if the initial data satisfies 
    \begin{align*}
        \Vert v_{in}, b_{in}\Vert_{H^N }= \eps \le \delta_{\alpha,r} {\vert \ln(\mu)\vert^{-(1+r)} } \mu^{\frac 1 3 },
    \end{align*}
    then the solution $(v,b)$ of \eqref{eq:vb} is in $ C_t H^N$ and satisfies the following: \\
    \textbf{Stability of velocity and magnetic field} 
        \begin{align*}
        \sup_{t\ge 0} \Vert v,b  \Vert_{H^{N} } &\lesssim \eps .
    \end{align*}    
     \textbf{Inviscid damping and Enhanced dissipation,} for all $t\ge 0$ we obtain 
    \begin{align*}
        \Vert v^x_{\neq},b^x_{\neq}  \Vert_{H^{N }}+ \langle t\rangle \Vert v^y, b^y \Vert_{H^{N-1  }}&\lesssim \eps  e^{-c\mu^{\frac 1 3 } t }.
    \end{align*}  
    
\end{theorem}

In the remainder of this section, we prove Theorem \ref{thm:MHD}. We distinguish between constant magnetic fields with only horizontal components $\alpha = \alpha_1 e_1$ and constant magnetic fields with vertical and prove the threshold. In both cases, we define an adapted velocity field $\tilde v$ and perform an energy estimate. Let $d\in \R^2\setminus \{0\}$ and $c_1>0$, we define the linear weight
\begin{align}\begin{split}
    M_{L,d,c_1} (k,\eta)  &=
    \exp\left( c_1^{-1} \int^t_{\infty} \tfrac 1 {\sqrt {1 +(d_1 + d_2(t-\frac \eta k ))^2 }}\right),\qquad \qquad  k\neq 0,  \label{Malpha}\\
     M_{L,d,c_1}(0,\eta) &=  M_{L,d,c_1}(1,\eta) .
\end{split}\end{align}
The weight $M_{L,d,c_1}$ is a generalisation of the standard weight which we obtain for $d=(0,1)$. We obtain the following properties: 
\begin{lemma}\label{lem:Mb} Let $d\in \R^2$ and $c_1>0$ such that $d_2 \neq 0$ and $(k,\eta), (l,\xi) \in \bbZ\times \R $, then $M_{L,d,c_1} $ is monoton increasing and satisfies 
    \begin{enumerate}[label=(\roman*)]
        \item $ 1\le M_{L,d,c_1} (k,\eta ) \le e^{2\vert d_2\vert ^{-1}\pi} .$
    \item For $k\neq 0$ and  $\vert l,\xi \vert \ge \vert k-l,\eta-\xi\vert$ it holds 
    \begin{align*}
       \vert {M_{L,d,c_1} (k,\eta)}-{M_{L,d,c_1} (k,\xi)} \vert  &\lesssim \tfrac {\vert \eta -\xi\vert } {\vert k\vert} 
    \end{align*}
    \item For $\vert l,\xi \vert \ge \vert k-l,\eta-\xi\vert$ it holds 
    \begin{align*}
        \frac {\vert l,\xi\vert}{t} \vert M_{L,d,c_1} (k,\eta) -M_{L,d,c_1}(l,\eta ) \vert &\lesssim \left( 1+\mu^{-\frac 13 } \vert \ln(\mu)\vert^{1+r} \sqrt{\frac{\p_t m_0 }{m_0} (k,\eta ) }\sqrt{\frac{\p_t m_0 }{m_0} (l,\xi ) } \right) \langle  k-l, \eta-\xi \rangle^5.
    \end{align*}
    \end{enumerate}
    In particular, $M_{L,d,c_1}$ is $\gamma=0$-admissible. 
\end{lemma}
The proof can be found in the appendix.

\subsection{Horizontal Constant Magnetic Field}
Let $\alpha = \alpha_1 e_1$, we prove the stability of
\begin{align}\begin{split}\label{eq:vb1}
    \partial_t v + v^y e_1 - 2\partial_x \Delta^{-1}_t  \nabla_t v^y  &=  \nu  \Delta_t v+ \alpha_1 \partial_x b  + (b\cdot \nabla_t b)- (v\cdot \nabla_t) v-\nabla_t \pi , \\
    \partial_t b - b^y e_1 \qquad \qquad \quad \quad  \ \ &= \kappa \Delta_t   b+ \alpha_1 \partial_x v  +(b\cdot \nabla_t) v -(v\cdot \nabla_t) b,\\
    \dive_t(v)=\dive_t(b) &= 0.
\end{split}\end{align}
 For small values of $\alpha_1$, due to the lack of interaction, the $-b^y e_1$ term leads to growth. We adapt to this effect by introducing the adapted velocity (cf. \cite{NiklasMHD2024})
\begin{align*}
    \tilde v &= v + \tfrac 1 {\alpha_1} \p_x^{-1} b^y_{\neq} e_1, & \tilde b&= b.
\end{align*}
Then \eqref{eq:vb1} changes to 
 \begin{align}\begin{split}
     \p_t \tilde v  &= \nu\Delta_t \tilde v+\tfrac 1 {\alpha_1}  (\kappa-\nu) \p_x^{-1} \Delta_t  b^y_{\neq } e_1 + \alpha_1 \p_x  b+  2\partial_x \Delta^{-1}_t  \nabla_t \tilde v^y\\
     &\quad + (b\cdot \nabla_t) b- (v\cdot \nabla_t) v-\nabla_t \pi+ e_1 \tfrac 1 {\alpha_1} \p_x^{-1} ((b\cdot \nabla_t) v^y -(v\cdot \nabla_t) b^y)_{\neq} ,\\
     \p_t \tilde b &=\kappa\Delta_t \tilde b + \alpha_1 \p_x \tilde  v  +(b\cdot \nabla_t) v -(v\cdot \nabla_t) b,\\
     \dive_t(\tilde v) &=\tfrac 1 {\alpha_1} b^y_{\neq} , \quad \dive_t( b) =0. \label{eq:Lintvb}
\end{split} \end{align}

For $\mu=\min(\nu,\kappa)$ we define the weights
\begin{align*}
    A(k,\eta )&= \langle k ,\eta  \rangle^{N} e^{ c \mu^{\frac 1 3}t\textbf{1}_{k\neq 0}}m^{-1}(k,\eta),\\
    m(k,\eta )&=  m_0(k,\eta ) M_L^b (k,\eta ) M_\nu(k,\eta )M_\kappa(k,\eta ),
\end{align*}
with $M_L^{b}$ defined as 
\begin{align*}
    M_L^{b }(k,\eta)&=M_{L,(0,1), \frac c {\alpha_1}}(k,\eta), 
\end{align*}
i.e. the weight \eqref{Malpha} with $d=e_2$ and $c_1= \frac c {\alpha_1}$.

For times $0\le t\le 1 $ the norm stays bounded by classical energy estimates,  $ \Vert  A (\tilde v,\tilde b) \Vert^2_{L^2} \le C_1 \eps^2$ for some $C_1>1$.  In the following, we prove that the energy $ \Vert  A (\tilde v,\tilde b) \Vert^2_{L^2}$ stays bounded for times $t\ge 1$. We do this by a Bootstrap method, for $1<t$, the bootstrap assumption states:
\begin{align}
\begin{split}\label{bootMHD}
   \Vert  A (\tilde v,\tilde b) \Vert^2_{L^2}  +\int_1^t \Vert \Lambda_t A   (\sqrt \nu\tilde v,\sqrt \kappa \tilde b)\Vert^2_{L^2} +\Vert \sqrt {\tfrac {\p_t m }m}  (\tilde v,\tilde b)  \Vert_{L^2}^2 d\tau&\le C_2\eps^2.
\end{split}
\end{align}

\begin{prop}\label{pro:BootMHD}
Let $(\tilde v,\tilde  b) $ be a solution to \eqref{eq:Lintvb} which satisfies the Bootstrap hypotheisis \eqref{bootMHD} for $1\le t$, then there exists a for a $C_3=C_3(C_2,N,  \alpha , c  ) $ such that 
\begin{align*} \Vert  A (\tilde v,\tilde b) \Vert^2_{L^2}  +\int_1^t \Vert \Lambda_t A   (\sqrt \nu\tilde v,\sqrt \kappa \tilde b)\Vert^2_{L^2} +\Vert \sqrt {\tfrac {\p_t m }m}  (\tilde v,\tilde b)  \Vert_{L^2}^2 d\tau&\le (C_1 + 3 cC_2 + C_3 \delta) \eps^2.
\end{align*}
\end{prop}
The proof of Theorem \ref{thm:MHD} from Proposition \ref{pro:BootMHD} is the same as for the Boussinesq equations. To prove this proposition, we use two Lemmas, 
\begin{lemma}\label{lem:MHDsplit}
    Let the bootstrap assumption \eqref{bootMHD} hold for $0\le t \le t^\ast$, then it holds that 
    \begin{align*}
        \p_t \Vert A (\tilde v,\tilde b)\Vert_{L^2}^2 +\Vert \Lambda_t A(\sqrt \nu \tilde v, \sqrt \kappa \tilde b)\Vert_{L^2}^2 + \Vert \sqrt{\tfrac {\p_t m } m }A(\tilde u ,\tilde b ) \Vert_{L^2 } ^2&\lesssim L +\sum_{(f,q ) \in \Gamma} NL_{f,q}  +LNL 
    \end{align*}
    with  $\Gamma =\{(f_1,f_2,q): f_i \in \{v,b\} , \ q \in \{\Lambda_t^{-1}\nabla^\perp_t \cdot v,\Lambda_t^{-1}\nabla^\perp_t \cdot b\}\}$. Where we denote \\
\textbf{Liner terms:}
\begin{align*}
    L&= c \mu^{\frac 1 3}\Vert A(\tilde v_{\neq},\tilde b_{\neq})\Vert_{L^2}^2 +\tfrac 1 {\alpha_1}  \langle A \tilde v^x ,  (\kappa-\nu) \p_x^{-1} \Delta_t A \tilde b^y_{\neq }\rangle + \tfrac 2 {\alpha_1} \langle \tilde b^y,\partial_x \Delta^{-1}_t   \tilde v^y \rangle 
\end{align*}
\textbf{Main nonliner term:} 
\begin{align*}
    NL_{f,q}&= \langle Af_1 , [A,\Lambda_t^{-1}\nabla^\perp q ] \cdot \nabla f_2\rangle 
\end{align*}
\textbf{Lower nonliner term:}
\begin{align*}
    LNL&=\frac 1{\alpha_1} \langle A\tilde v^x_{\neq}, A\p_x^{-1} ((b\cdot \nabla_t) v^y -(v\cdot \nabla_t) b^y)_{\neq}\rangle  \\
    &+\frac 1 {\alpha_1}\langle A\p_x^{-1} b_{\neq}^y, A((b\cdot \nabla_t) v^x-(v\cdot \nabla_t) b^x)_{\neq}\rangle \\
    &+\frac 1 {\alpha_1}  \langle A b^y_{\neq}, A\pi_{\neq}\rangle
\end{align*}
\end{lemma}
\begin{proof}
    Lemma \ref{lem:MHDsplit} follows by direct calculations. 
\end{proof}

\textbf{Proof of Proposition \ref{pro:BootMHD}:} By Lemma \ref{lem:MHDsplit} we only need to bound 
\begin{align}
    \int_1^t L d\tau &\le 3c C_2 \eps^2, \label{eq:BLMHD}\\
    \int_1^t NL_f d\tau &\lesssim \delta \eps^2, \label{eq:NLMHD}\\
    \int_1^t LNL d\tau &\lesssim \delta \eps^2. \label{eq:LNLMHD}
\end{align}
By definition of $(\tilde v , \tilde b)$ the bootstrap assumption \eqref{bootMHD} also holds for $(v,b)$ and $(\Lambda_t^{-1}\nabla_t^{\perp}\cdot v,\Lambda_t^{-1}\nabla_t^{\perp}\cdot b)$ with a different constant. By Lemma \ref{lem:M1},  $M_L^{b,1}M_{\min(\nu,\kappa)}$ is admissible to apply Theorem \ref{thm:main} with $\gamma=\tilde \gamma =0$ and $(f,q)\in \Gamma$ and therefore we obtain \eqref{eq:NLMHD}.

\textbf{Linear estimates:} We use the identity $\p_x^{-1} \Delta_t  \tilde b^y_{\neq }= \nabla^\perp_t \cdot \tilde b_{\neq } $ to infer 
\begin{align*}
    \langle A \tilde v^x ,  (\kappa-\nu) \p_x^{-1} \Delta_t A \tilde b^y_{\neq }\rangle&=(\kappa-\nu)\langle A \tilde v^x ,   A\nabla^\perp_t \cdot \tilde b_{\neq }\rangle\\
    &\lesssim \nu \Vert \Lambda_t A\tilde v \Vert _{L^2 }\Vert  A\tilde b_{\neq } \Vert _{L^2 }+\kappa \Vert  A\tilde v_{\neq }  \Vert _{L^2 }\Vert  A\Lambda_t\tilde b_{\neq } \Vert _{L^2 }
\end{align*}
and thus integrating in time, using \eqref{bootMHD}, Lemma \ref{cor:L2} and $\kappa^3 \le \nu \le \kappa^{-\frac 13} $ yields 
\begin{align*}
    \int \langle A \tilde v^x ,  (\kappa-\nu) \p_x^{-1} \Delta_t A \tilde b^y_{\neq }\rangle d\tau \lesssim c\left(\nu^{\frac 1 2 } \kappa^{-\frac 16}+\kappa^{\frac 1 2 } \nu^{-\frac 16}\right) C_2\eps^2 \le 2 c C_2\eps^2 
\end{align*}
For the second term, by definition of $M_L^b$ we estimate 
\begin{align*}
    \tfrac 1{\alpha_1} \langle b^y,2\partial_x \Delta^{-1}_t   \tilde v^y \rangle &\le c \Vert \sqrt{\tfrac {\p_t M_L^{b}}{ M_L^{b}}}(\tilde v,\tilde b ) \Vert_{L^2 }^2 
\end{align*}
and thus after integrating in time 
\begin{align*}
    \int_1^t\tfrac 1 {\alpha_1} \langle b^y,2\partial_x \Delta^{-1}_t   \tilde v^y \rangle d\tau &\le c C_2 \eps^2.
\end{align*}

\textbf{Lower order estimates:} Note that $v^y_==b^y_==0 $, therefore we estimate directly 
\begin{align*}
     LNL_1:&=\frac 1 {\alpha_1}\langle A\tilde v^x_{\neq} , A\p_x^{-1} ((b\cdot \nabla_t) v^y -(v\cdot \nabla_t) b^y)_{\neq}\rangle \\
     &\lesssim \Vert A\tilde v^x_{\neq} \Vert_{L^2 } \left(\Vert v,b\Vert_{H^6}\Vert\p_x^{-1} \Lambda_t A (v^y,b^y) \Vert_{L^2 }+ \Vert A(v,b)\Vert_{L^2}\Vert \Lambda_t   (v^y,b^y) \Vert_{H^6 } \right)\\
    &\lesssim \Vert A(v_{\neq },b_{\neq}) \Vert_{L^2 } \Vert A(v,b) \Vert_{L^2 }\Vert\p_x^{-1} \Lambda_t A (v^y,b^y) \Vert_{L^2 }.
\end{align*}
For $a \in \{v,b\}$ it holds that 
\begin{align*}
    \Lambda_ta^y = \p_x \Lambda^{-1}_t \nabla^\perp_t\cdot  a_{\neq}
\end{align*}
and therefore 
\begin{align*}
    \Vert\p_x^{-1} \Lambda_t A (v^y,b^y) \Vert_{L^2 }\lesssim \Vert  A (v_{\neq },b_{\neq }) \Vert_{L^2 }.
\end{align*}
After integrating in time, we estimate 
\begin{align*}
    \int_1^t LNL_1 d\tau &\lesssim \Vert A(v_{\neq },b_{\neq}) \Vert_{L^2 L^2 }^2 \Vert A(v,b) \Vert_{L^\infty L^2 }\lesssim \mu^{-\frac 13 } \eps^3 \le \delta \eps^2.
\end{align*}
Similar we estimate 
\begin{align*}
    LNL_2:&=\tfrac 1 {\alpha_1}\langle A\p_x^{-1} b_{\neq}^y, A((b\cdot \nabla_t) v-(v\cdot \nabla_t) b)\rangle =\tfrac 1 {\alpha_1}\langle A\p_x^{-1} \nabla_t b_{\neq}^y, A((b  v^x-v b^x)_{\neq}\rangle  \\
    &\lesssim \Vert  A\p_x^{-1} \nabla_t b_{\neq}^y\Vert_{L^2 }  \Vert  A(v,b)\Vert_{L^2}\Vert  A(v_{\neq},b_{\neq})\Vert_{L^2}\lesssim \Vert  A b_{\neq}\Vert_{L^2 }  \Vert  A(v,b)\Vert_{L^2}\Vert  A(v_{\neq},b_{\neq})\Vert_{L^2}
\end{align*}
and so 
\begin{align*}
    \int_1^t LNL_2 d\tau &\lesssim  \delta \eps^2.
\end{align*}

Since $\pi =\Delta^{-1}_t \nabla_t   ((b\cdot \nabla_t) b- (v\cdot \nabla_t) v)=\Delta^{-1}_t (\nabla_t\otimes \nabla_t) : (b\otimes  b- v\otimes  v) $ we infer 
\begin{align*}
     LNL_3:&=\frac 1 \alpha_1  \langle A b^y_{\neq},A \pi_{\neq}\rangle= \frac 1 \alpha_1  \langle A \Lambda^{-1}_t \nabla^\perp _t b_{\neq},\p_x \Lambda^{-1}_t A \pi_{\neq}\rangle\lesssim \Vert A(v_{\neq},b_{\neq}) \Vert_{L^2}^2 \Vert A(v,b) \Vert_{L^2}.
\end{align*}
Therefore, after integrating in time 
    \begin{align*}
    \int_1^t LNL_3 d\tau &\lesssim \mu^{-\frac 13 } \eps^3\le  \delta \eps^2. 
\end{align*}
Since $LNL=LNL_1+LNL_2+LNL_3$ this yields the estimate \eqref{eq:LNLMHD}. Then Proposition \ref{pro:BootMHD} follows and so Theorem \ref{thm:MHD} for the case of horizontal constant magnetic field.

\subsection{Constant Magnetic Field with Vertical Component}
If $\alpha_2 \neq 0$,  we prove the stability of
\begin{align*}
    \partial_t v + v^y e_1 - 2\partial_x \Delta^{-1}_t  \nabla_t v^y  &=  \nu  \Delta_t v+ (\alpha \cdot \nabla_t)  b  + (b\cdot \nabla_t) b- (v\cdot \nabla_t) v-\nabla_t \pi , \\
    \partial_t b - b^y e_1 \qquad \qquad \quad \quad  \ &= \kappa \Delta_t   b+ (\alpha \cdot \nabla_t)  v  +(b\cdot \nabla_t) v -(v\cdot \nabla_t) b,\\
    \dive_t(v)=\dive_t(b) &= 0.
\end{align*}
Before proving stability, we define unknowns that adapt to the linear dynamics,
\begin{align*}
    \tilde v &= v + \tfrac {\alpha \cdot \nabla_t} {\langle \alpha \cdot \nabla_t \rangle^2 } b^y_{\neq} e_1, &
    \tilde b&= b.
\end{align*}
For the adapted velocity, this yields the equations
\begin{align*}
    \p_t \tilde v &= \p_t v + \p_t (\tfrac {\alpha \cdot \nabla_t} {\langle \alpha \cdot \nabla_t \rangle^2 } )b^y_{\neq} e_1 +\tfrac {\alpha \cdot \nabla_t} {\langle \alpha \cdot \nabla_t \rangle^2 } \p_t b^y_{\neq} e_1\\
    &=-v^y e_1 + 2\partial_x \Delta^{-1}_t  \nabla_t v^y + \nu  \Delta_t v+ (\alpha \cdot \nabla_t)  b  + (b\cdot \nabla_t) b- (v\cdot \nabla_t) v-\nabla_t \pi\\
&\quad -\alpha_2\p_x \tfrac {1-(\alpha\cdot  \nabla_t)^2}{\langle \alpha\cdot\nabla_t  \rangle^4}b^y e_1+ \tfrac {\alpha \cdot \nabla_t} {\langle \alpha \cdot \nabla_t \rangle^2 }e_1(\kappa \Delta_t   b^y+ (\alpha \cdot \nabla_t)  v^y  +(b\cdot \nabla_t) v^y -(v\cdot \nabla_t) b^y)_{\neq}\\
&=\nu\Delta_t \tilde v+ (\kappa-\nu)    \tfrac {\alpha \cdot \nabla_t} {\langle \alpha \cdot \nabla_t \rangle^2 } \Delta_t  \tilde b^y_{\neq } e_1 + (\alpha \cdot \nabla_t)  \tilde b+  2\partial_x \Delta^{-1}_t  \nabla_t \tilde v^y- \tfrac {1} {\langle \alpha \cdot \nabla_t \rangle^2 }\tilde v^y e_1\\
&\quad -\alpha_2\p_x \tfrac {1-(\alpha\cdot  \nabla_t)^2}{\langle \alpha\cdot\nabla_t \rangle^4}\tilde b^y  e_1+ (b\cdot \nabla_t) b- (v\cdot \nabla_t) v-\nabla_t \pi+ e_1  \tfrac {\alpha \cdot \nabla_t} {\langle \alpha \cdot \nabla_t \rangle^2 } ((b\cdot \nabla_t) v^y -(v\cdot \nabla_t) b^y)_{\neq}, 
\end{align*}
where we used that  $\p_t (\tfrac {\alpha \cdot \nabla_t} {\langle \alpha \cdot \nabla_t \rangle^2 } )=-\alpha_2\p_x \tfrac {1-(\alpha\cdot  \nabla_t)^2}{\langle \alpha\cdot\nabla_t \rangle^4} $ and $1-\frac{(\alpha \cdot \nabla_t)^2 }{\langle \alpha \cdot \nabla_t \rangle^2 }=\frac{1 }{\langle \alpha \cdot \nabla_t \rangle^2 }$. The magnetic field satisfies the equation
\begin{align*}
    \p_t \tilde b &= b^y e_1 + \kappa \Delta_t   b+ (\alpha \cdot \nabla_t)  (\tilde v - \tfrac {\alpha \cdot \nabla_t} {\langle \alpha \cdot \nabla_t \rangle^2 } b^y_{\neq} e_1) +(b\cdot \nabla_t) v -(v\cdot \nabla_t) b\\
    &=\kappa\Delta_t \tilde b + (\alpha \cdot \nabla_t) \tilde  v -\tfrac {1} {\langle \alpha \cdot \nabla_t \rangle^2 }b^ye_1+(b\cdot \nabla_t) v -(v\cdot \nabla_t) b.
\end{align*}

Therefore, \eqref{eq:vb} changes to the equations 
 \begin{align}\begin{split}
     \p_t \tilde v  &= \nu\Delta_t \tilde v+ (\kappa-\nu)    \tfrac {\alpha \cdot \nabla_t} {\langle \alpha \cdot \nabla_t \rangle^2 } \Delta_t  \tilde b^y e_1 + (\alpha \cdot \nabla_t)  \tilde b+  2\partial_x \Delta^{-1}_t  \nabla_t \tilde v^y- \tfrac {1} {\langle \alpha \cdot \nabla_t \rangle^2 }\tilde v^y e_1\\
     &\quad-\alpha_2\p_x \tfrac {1-(\alpha\cdot  \nabla_t)^2}{\langle \alpha\cdot \nabla_t \rangle^4}\tilde b^y e_1+ (b\cdot \nabla_t) b- (v\cdot \nabla_t) v-\nabla_t \pi+ e_1  \tfrac {\alpha \cdot \nabla_t} {\langle \alpha \cdot \nabla_t \rangle^2 } ((b\cdot \nabla_t) v^y -(v\cdot \nabla_t) b^y)_{\neq} ,\\
     \p_t \tilde b &=\kappa\Delta_t \tilde b + (\alpha \cdot \nabla_t) \tilde  v + \tfrac {1} {\langle \alpha \cdot \nabla_t \rangle^2 }\tilde b^ye_1+(b\cdot \nabla_t) v -(v\cdot \nabla_t) b,\\
     &\dive_t( \tilde v) =\p_x \tfrac {\alpha \cdot \nabla_t} {\langle \alpha \cdot \nabla_t \rangle^2 } \tilde b^y , \quad \dive_t ( \tilde b) =0. \label{eq:Lintvb2}
\end{split} \end{align}
 Let $c_2 =\frac1{10  \max ( 1 , \alpha_2 )}$ and we use the linear weight 
\begin{align*}\begin{split}
    M_\alpha (k,\eta)  &= M_{L, \alpha,c_2} (k,\eta)\cdot M_{L, (0,1),c_2} (k,\eta)
\end{split}\end{align*}
i.e. two times the weight \eqref{Malpha} with $d=\alpha$ and $d=(0,1)$. By Lemma \ref{lem:Mb}, this weight is $\gamma=0$-admissible.  For $\mu=\min(\nu,\kappa)$ we define the weights
\begin{align*}
    A(k,\eta )&= \langle k ,\eta  \rangle^{N} e^{ c \mu^{\frac 1 3}t\textbf{1}_{k\neq 0}}m^{-1}(k,\eta),\\
    m(k,\eta )&=  m_0(k,\eta ) M_\alpha(k,\eta) M_\nu(k,\eta )M_\kappa(k,\eta ). 
\end{align*}
The proof of stability is analogous to the previous stability results. For $t>1$ we write the bootstrap hypothesis 
\begin{align}
\begin{split}\label{bootMHD2}
   \Vert  A (\tilde v,\tilde b) \Vert^2_{L^2}  +\int_1^t \Vert \Lambda_t A   (\sqrt\nu\tilde v,\sqrt\kappa \tilde b)\Vert^2_{L^2} +\Vert \sqrt {\tfrac {\p_t m }m}  (\tilde v,\tilde b)  \Vert_{L^2}^2 d\tau&\le C_2\eps^2.
\end{split}
\end{align}
From that, we can establish the equivalent of Lemma \ref{lem:MHDsplit} and Proposition \ref{pro:BootMHD}. Here we see that the main nonlinear term is the same as in  Lemma \ref{lem:MHDsplit} and the lower nonlinearities are simpler to estimate than the one in  Lemma \ref{lem:MHDsplit} since $\tfrac {\alpha \cdot \nabla_t} {\langle \alpha \cdot \nabla_t \rangle^2 }$ introduces additional decay. To omit repeating steps, we only estimate the linear terms. For the linear terms, we obtain  
\begin{align*}
    L &\le c\mu^{\frac 13 } \Vert A(\tilde v_{\neq}, \tilde b_{\neq} )  \Vert_{L^2}  +\vert \kappa-\nu\vert\vert  \langle A\tilde v,   \tfrac {\alpha \cdot \nabla_t} {\langle \alpha \cdot \nabla_t \rangle^2 } \Delta_t  b^y_{\neq }\rangle\vert  + \vert  \langle A(\tilde v^x, \tilde b^x), \tfrac {1} {\langle \alpha \cdot \nabla_t \rangle^2 }(\tilde v^y, \tilde b^y)\rangle \vert \\
    &+2 \vert \langle A\p_x  \tfrac {\alpha \cdot \nabla_t} {\langle \alpha \cdot \nabla_t \rangle^2 } b^y_{\neq}, \partial_x \Delta^{-1}_t  A \tilde v^y\rangle\vert +\alpha_2 \vert \langle A  \tilde v^x  , \p_x \tfrac {(\alpha\cdot  \nabla_t)^2-1}{\langle \alpha\cdot\nabla_t  \rangle^4} Ab^y \rangle \vert  \\
    &= L_1 +L_2 +L_3 +L_4 +L_5.
\end{align*}
The $L_1$ and $L_4$ can be estimated as in the previous subsection and we omit the details. We estimate the $L_2$ term by 
\begin{align*}
    L_2 &=\vert \kappa-\nu\vert \vert \langle A\tilde v,   \tfrac {\alpha \cdot \nabla_t} {\langle \alpha \cdot \nabla_t \rangle^2 } \Delta_t  b^y_{\neq }\rangle\vert\le(\kappa+\nu)\vert \langle A\p_x \tfrac {\alpha \cdot \nabla_t} {\langle \alpha \cdot \nabla_t \rangle^2 } \tilde v,  \nabla_t^\perp\cdot  b_{\neq }\rangle\vert \\
 &\le\kappa\Vert \sqrt {\tfrac {\p_t M_\alpha }{M_\alpha }} A\tilde v\Vert_{L^2}\Vert  A\nabla_t \tilde b\Vert_{L^2} +\nu\Vert \sqrt {\tfrac {\p_t M_\alpha }{M_\alpha }} A\tilde b\Vert_{L^2}\Vert  A\nabla_t \tilde v\Vert_{L^2}
\end{align*}
and therefore after integrating in time
\begin{align*}
    \int_1^t L_2 d \tau \lesssim \sqrt {c_2}(\kappa^{\frac 12} +\nu^{\frac 1 2 })C_2 \eps^2.
\end{align*}
For $L_3$, we estimate 
\begin{align*}
    L_3 
    &=\vert  \langle A(\tilde v^x, \tilde b^x), \tfrac {1} {\langle \alpha \cdot \nabla_t \rangle^2 }(\tilde v^y, \tilde b^y)\rangle \vert \le c_2 \Vert \sqrt {\tfrac {\p_t M_\alpha }{M_\alpha }} A( \tilde v,\tilde b )\Vert_{L^2}^2.
\end{align*}
Therefore, after integrating in time 
\begin{align*}
    \int_1^t  L_3 d\tau  &\le c_2 C_1 \eps^2 .
\end{align*}
For $L_5$ we obtain
\begin{align*}
   L_5 =\vert\langle A  \tilde v^x  , \alpha_2\p_x \tfrac {(\alpha\cdot  \nabla_t)^2-1}{\langle \alpha\cdot\nabla_t  \rangle^4} Ab^y \rangle \vert
    &\le 2 c_1 \alpha_2 \Vert  A \sqrt{\tfrac{\p_t M_\alpha }{M_\alpha}}  \tilde v\Vert_{L^2 }  \Vert  A \sqrt{\tfrac{\p_t M_\alpha }{M_\alpha}} \tilde b \Vert_{L^2}.
\end{align*}
After integrating in time we infer
\begin{align*}
    \int_1^tL_5 d\tau \le \alpha_2 c_2 C_1\eps^2.
\end{align*}
Thus, since $c_2$ is small enough, we obtain linear estimates with an admissible weight. The main nonlinear and lower nonlinear terms work the same as in the case of constant magnetic field which is enough to establish the equivalent Proposition \ref{pro:BootMHD} for equation \eqref{eq:Lintvb2}. This yields Theorem \ref{thm:MHD} and then finally Theorem \ref{thm:thres}.

\subsection*{Data availability}
No data was used for the research described in the article.

\subsection*{Acknowledgments}

N. Knobel is supported by ERC/EPSRC Horizon Europe Guarantee EP/X020886/1. The author declares that he has no conflict of interest.

The author wants to thank Michele Coti Zelati for his useful comments on the introduction.

\appendix
\section{The Navier-Stokes Threshold}\label{sec:NS}
In this section, we use Theorem \ref{thm:main} to prove the threshold for Navier-Stokes equations in Theorem \ref{thm:thres}. The Navier-Stokes equation in velocity formulation states 
\begin{align*}\begin{split}
    \partial_t V +V\cdot\nabla V  + \nabla \Pi  &= \nu \Delta V,\\
    \dive (V) &= 0.
\end{split}\end{align*}
For the  perturbative unknown $v(t,x,y) =V(t,x+yt, y ) - V_s(x,y)$ we obtain 
\begin{align*}
\begin{split}
    \partial_t v+ e_1 v^y   + (v \cdot\nabla_t)v  + \nabla_t \pi  &= \nu \Delta_t v, \\
    \dive_t (v) &= 0.
\end{split}
\end{align*}
which yields for the vorticity $w = \nabla^\perp_t\cdot  v $ the equation

\begin{align}\begin{split}
    \partial_t w  + (\overline v \cdot\nabla) w &= \nu \Delta_t w,\\
    \overline v &=\Delta_t^{-1}\nabla^{\perp}  w,\label{eq:NS}
\end{split}\end{align}
For this equation, we obtain the following stability:

\begin{theorem}[Navier-Stokes part of Theorem \ref{thm:thres}]\label{thm:NS}
    Let $N\ge 12$, then there exist $c, \delta >0$ such that if the initial data satisfies 
    \begin{align*}
        \Vert w_{in}\Vert_{H^N}= \eps \le \delta \nu^{\frac 13 } 
    \end{align*}
    then the global in time solution $w$ to \eqref{eq:NS}  is in $ C_t H^N$ and satisfies the following:\\
     \textbf{Stability of the vorticity}
        \begin{align*}
        \sup_{t\ge 0}\Vert w \Vert_{H^{N} } &\lesssim \eps .
    \end{align*}    
     \textbf{Inviscid damping and Enhanced dissipation estimates }
    \begin{align*}
        \Vert w_{\neq}  \Vert_{H^{N }}+\langle t\rangle^{\frac 12 }\Vert v^x_{\neq} \Vert_{H^{N-\frac 1 2  }}+ \langle t\rangle^{\frac 32 }\Vert v^y \Vert_{H^{N-\frac 3 2  }}&\lesssim \eps  e^{-c\nu ^{\frac 1 3 } t }.
    \end{align*}  
\end{theorem}
 
\begin{proof} The proof is a direct application of Theorem \ref{thm:main} for $\gamma=1$ and $\tilde \gamma =0$. We use the time-dependent Fourier multiplier 
    \begin{align*}
    A(k,\eta )&= \langle k ,\eta  \rangle^{N} e^{ c \mu^{\frac 1 3}t\textbf{1}_{k\neq 0}}m^{-1}(k,\eta),\\
    m(k,\eta )&=  m_1(k,\eta )  M_\nu(k,\eta ).
\end{align*}
This corresponds to the linear weight $M_L=1$, which is admissible in sense of Definition \ref{def:admiss}. By a classical energy argument, obtain $\Vert  A w \Vert^2_{  L^2}\le C_1 \eps^2 $ for some $C_1>0$ and $t\in [0,1]$. We prove by bootstrap that for $t>1$, it holds that
    \begin{align}
        \Vert  A w \Vert^2_{  L^2}  +c_1\int_1^t \nu \Vert \nabla_t A  w\Vert^2_{L^2} +\Vert \sqrt {\tfrac {\p_t m }m} A  w\Vert_{L^2}^2 d\tau&\le C_2\eps^2, \label{bootw}
    \end{align}
    for some $C_1>2(1+\frac {\Vert A w|_{t=1}\Vert_{L^2}}{\Vert w_{in}\Vert_{H^N}})$. For the sake of contradiction, we assume that there exists a maximal time  $T$ such that the bootstrap assumption \eqref{bootw} holds. We estimate 
    \begin{align*}
        \frac12 \p_t \Vert  A w \Vert^2_{ L^2}  &= \langle A w, (\p_t A) w\rangle + \nu\langle A w,  A\Delta_t w \rangle - \langle Aw , A((\overline v\cdot  \nabla) w )\rangle 
    \end{align*}
    and by Lemma \ref{lem:M2}  we obtain 
    \begin{align*}
        \langle A w, \p_t A w\rangle &\le c\nu^{\frac 13 } \Vert  A w_{\neq} \Vert_{L^2}^2- \Vert \sqrt {\tfrac {\p_t m }m} A  w\Vert_{L^2}^2\le -\frac 12 \Vert \sqrt {\tfrac {\p_t m }m} A  w\Vert_{L^2}^2+ \frac 12 \nu \Vert \nabla_t A  w\Vert^2_{L^2}.
    \end{align*} 
    By partial integration we obtain $\langle Aw , (\overline v\cdot \nabla) A w )\rangle =0$ and so 
\begin{align*}
    \langle Aw , A((\overline v\cdot\nabla ) w )\rangle &=\langle Aw , [A,\overline v ] \cdot \nabla w )\rangle .
\end{align*}
    Therefore, we estimate 
     \begin{align*}
    \Vert  A w \Vert^2_{L^\infty L^2}  &+ \int_1^t \nu \Vert \nabla_t A  w\Vert^2_{L^2} +\Vert \sqrt {\tfrac {\p_t m }m} A  w\Vert_{L^2}^2 d\tau\\
    & \le \Vert A w_{in}\Vert_{L^2}^2  +2\int_0^t \langle A w , [A, \nabla^\perp \Lambda^{-2}_t w]\cdot \nabla w \rangle d\tau .
    \end{align*}
 By Theorem \ref{thm:main} we obtain the estimate
    \begin{align*}
        \int_1^t \vert \langle A w , [A, \nabla^\perp \Lambda^{-2}_t w]\cdot \nabla w \rangle\vert  \le  C_3 \delta \eps^2. 
    \end{align*}
    for some $ C_3>0$. Therefore, for $\delta$ small enough, estimate \eqref{bootw} holds with $<$. By local existence, this contradicts the maximality. Thus, we infer \eqref{bootw} holds globally in time. The inviscid damping and enhanced dissipation estimates are a consequence of the definition of $A$ and $v$.    
\end{proof}

\section{Estimates for the Linear Weights}\label{sec:linweight}
\begin{proof}[Proof of Lemma \ref{lem:M2} ] The properties i) and ii) hold by definition. Property iii) can be obtained by applying the mean value theorem and using that $\p_\eta M_\mu =-\frac 1k \p_t  M_\mu$. To prove iv) we consider the case when $k,l\neq 0$, since $M(t,0,\eta)=M(t,1,\eta) $. We use $\vert e^{\vert x\vert}-1\vert\le \vert x\vert e^{\vert x\vert}$ to estimate 
\begin{align*}
    \vert M_\mu (k,\eta) -M_\mu(l,\xi)\vert &\lesssim \left\vert  \exp\left(\int_0^t \tfrac {\mu^{\frac 1 3 }}{1+ \mu^{\frac 2 3 }\vert \tau -\frac \eta k \vert^{2} }-\tfrac {\mu^{\frac 1 3 }}{1+ \mu^{\frac 2 3 }\vert \tau -\frac \xi l \vert^{2} }d\tau \right)-1\right\vert\\
    &\lesssim  \mu^{\frac 13}\left\vert\int_0^t \tfrac {1}{1+ \mu^{\frac 2 3 }\vert \tau-\frac \eta k \vert^{2} }-\tfrac {1}{1+ \mu^{\frac 2 3 }\vert \tau-\frac \xi l \vert^{2} }d\tau \right\vert .
\end{align*}
We distinguish between $\vert \frac \xi l\vert  \le 2 t  $ and $\vert \frac \xi l\vert  \ge 2 t  $ cases. Let $\vert \frac \xi l \vert \ge 2 t  $, then we obtain $\vert l,\xi\vert \le 2 \vert l,\xi-lt \vert$.
Furthermore we estimate 
\begin{align*}
    \left\vert\int_0^t \tfrac {1}{1+ \mu^{\frac 2 3 }\vert \tau-\frac \eta k \vert^{2} }-\tfrac {1}{1+ \mu^{\frac 2 3 }\vert \tau-\frac \xi l \vert^{2} }d\tau \right\vert \le t .
\end{align*}
Combining these estimates, we infer
\begin{align*}
    \frac {\vert l,\xi\vert }t\vert M(k,\eta)-M(l,\xi )  \vert\le 2\mu^{\frac 1 3 }\vert l,\xi-lt \vert. 
\end{align*}

For the case $\vert \frac \xi l\vert  \le 2t $ we write 
\begin{align*}
    \int_0^t \tfrac {1}{1+ \mu^{\frac 2 3 }\vert \tau-\frac \eta k \vert^{2} }-\tfrac {1}{1+ \mu^{\frac 2 3 }\vert \tau-\frac \xi l \vert^{2} }d\tau&= \int_{-\frac \eta k }^{t-\frac \eta k }\tfrac {1}{1+ \mu^{\frac 2 3 }s^{2} }ds-\int_{-\frac \xi l }^{t-\frac \xi l }\tfrac {1}{1+ \mu^{\frac 2 3 }s^{2} }ds\\
&=\int_{t-\frac \xi l  }^{t-\frac \eta k }\tfrac {1}{1+ \mu^{\frac 2 3 }s^{2} }ds-\int_{-\frac \eta k }^{t-\frac \xi l }\tfrac {1}{1+ \mu^{\frac 2 3 }s^{2} }ds.
\end{align*}

Therefore, we estimate 
\begin{align*}
    \left \vert\int_0^t \tfrac {1}{1+ \mu^{\frac 2 3 }\vert \tau-\frac \eta k \vert^{2} }-\tfrac {1}{1+ \mu^{\frac 2 3 }\vert \tau-\frac \xi l \vert^{2} }d\tau \right \vert &\le 2 \vert \frac \eta k -\frac \xi l\vert \le 2 \frac{\xi}{l^2}\langle \eta-\xi,k-l\rangle^2  
\end{align*}
and so we obtain 
\begin{align*}
    \frac {\vert l,\xi\vert }t\vert M(k,\eta)-M(l,\xi )  \vert &\lesssim   \mu^{\frac 13 } \frac{\vert l, \xi\vert^2 }{l^2 t }\langle \eta-\xi,k-l\rangle^2 \lesssim   \mu^{\frac 13 } t \langle \eta-\xi,k-l\rangle^2. 
\end{align*}

\end{proof}

\begin{proof}[Proof of Lemma \ref{lem:M1}]
The properties i) and ii) are a direct consequence of the definition of $M_L^\theta $. To prove iii), we use that 
\begin{align*}
    \vert M_L^\theta (k,\eta) - M_L^\theta (l,\xi)\vert\le \vert M_L^\theta (k,\eta) - M_L^\theta (k,\xi)\vert+\vert M_L^\theta (k,\xi) - M_L^\theta (l,\xi)\vert
\end{align*}
and estimate them separately. For  $\vert M_L^\theta (k,\eta) - M_L^\theta (k,\xi)\vert$, by mean value theorem for Lipschitz functions we infer 
\begin{align*}
    \vert M_L^\theta (k,\eta) - M_L^\theta (k,\xi)\vert &\le \vert \p_\eta M_L^\theta(k,\tilde \eta) \vert  \vert \eta-\xi\vert . 
\end{align*}
for some $\tilde \eta \in [\xi,\eta] $. We estimate 
\begin{align*}
    \p_\eta M_L^\theta(k,\tilde \eta)&\le \frac 1 k \left(c^3 \mu^{\frac 12 } +{\langle t-\tfrac {\tilde \eta } k\rangle^{-\frac 32 }}\right) M_L^\theta(k,\tilde \eta)
\end{align*}
in the Lipschitz sense. Thus we obtain 
\begin{align*}
     \tfrac {\vert l,\xi\vert}{t^{\frac 32}}\vert M_L^\theta (k,\eta) - M_L^\theta (k,\xi)\vert
     &\lesssim \tfrac {\vert l,\xi\vert}{t^{\frac 32}k}(\mu^{\frac 12 }+ {\langle t-\tfrac {\tilde \eta } k\rangle^{-\frac 32 }})\vert \eta-\xi\vert \\
     &\lesssim( t^{-\frac 32 } \mu^{\frac 12 }\vert l,\xi-lt \vert + t^{-\frac 12 } )\langle k-l, \eta-\xi\rangle^2  . 
\end{align*}

For $\vert M_L^\theta (k,\xi) - M_L^\theta (l,\xi)\vert $, we first show that if $M_L^\theta(k,\xi) -M_L^\theta(l,\xi) \neq 0$, then it holds that 
\begin{align}
    \vert \tfrac \xi l \vert \lesssim  ( t + \mu^{-\frac 1 6} )\vert k-l\vert\label{eq:T31add}
\end{align} 

If  $\min(\vert t-\frac \xi k \vert, \vert t-\frac \xi l \vert) \le 2c^{-1} \mu^{-\frac 1 6} $, then \eqref{eq:T31add} holds. Now let $\min(\vert t-\frac \xi k \vert, \vert t-\frac \xi l \vert) \ge 2c^{-1} \mu^{-\frac 1 6} $, then it holds $M_L^\theta (k,\xi ) - M_L^\theta(l,\xi)=0$ if either $t-\frac \xi k,t-\frac \xi l\le -2c^{-1}\mu^{-\frac 1 6}$ or  $t-\frac \xi k,t-\frac \xi l\ge \mu^{-\frac 1 6}$ or $k=l$. For the case  $t-\frac \xi k  \le \pm 2c^{-1}\mu^{-\frac 1 6} \le  t-\frac \xi l $ there exists a $k'\in (k,l)$ such that $t-\frac \xi {k'}=0$ and we infer \eqref{eq:T31add} since all caves are covered. 

To estimate $\vert M_L^\theta (k,\xi) - M_L^\theta (l,\xi)\vert$, by definition it is sufficient to consider $k,l\neq 0$ and since we can replace $(l,\xi)$ and $(k,\eta)$ it is sufficient to cosider the case ${ \langle t-\frac \xi k \rangle }\ge { \langle t-\frac \xi l \rangle }$. Since $M_1$ is bounded from below and above and using  $\vert e^x-1\vert\le \vert x\vert e^{\vert x\vert }$, we infer 
\begin{align*}
    \vert M_L^\theta (k,\xi) - M_L^\theta (l,\xi)\vert
    &\lesssim \left \vert\int_{-\infty}^t \textbf{1}_ {\vert \tau- \frac \eta k \vert\le 2\mu^{-\frac 1 6 }c^{-1}} \tfrac 1 {\langle\frac{\eta}{k}- \tau\rangle^{3 }} d\tau-\int_{-\infty}^t \textbf{1}_ {\vert \tau- \frac \xi k \vert\le 2\mu^{-\frac 1 6 }c^{-1}} \tfrac 1 {\langle\frac{\xi}{k}- \tau\rangle^{3 }} d\tau \right\vert \\
    &\le \left \vert \frac {t-\frac \xi k }{\langle t-\frac \xi k \rangle }- \frac {t-\frac \xi l }{\langle t -\frac \xi l \rangle }  \right\vert=\left \vert \frac {(t-\frac \xi k)\langle t -\frac \xi l \rangle- \langle t-\frac \xi k \rangle (t-\frac \xi l) }{\langle t-\frac \xi k \rangle\langle t -\frac \xi l \rangle }\right\vert\\
    &\lesssim 2 \left \vert \frac {(t-\frac \xi k)-(t-\frac \xi l) }{\langle t-\frac \xi k \rangle }\right\vert=  \vert \xi \vert \left\vert\frac 1 l-\frac 1 k \right\vert \frac 1{\langle t-\frac \xi k \rangle }= \left \vert \frac  \xi{lk } \right\vert \left\vert k-l \right\vert \frac 1{\langle t-\frac \xi k \rangle }. 
\end{align*}

Therefore, by using \eqref{eq:T31add} and Lemma \ref{lem:m1} we estimate 
\begin{align*}
    \vert M_L^\theta(k,\xi)&-M_L^\theta(l,\xi)\vert  \tfrac {\vert l,\xi\vert } {t^{\frac 32 }}\le \vert {k}-l \vert\vert  \tfrac {\xi }{ {k}  {l}}\vert  \tfrac 1{ \langle t-\frac \xi { k} \rangle }  \tfrac {\vert l,\xi\vert } {t^{\frac 32 }} \\
    &\lesssim  (\sqrt t+  \mu^{-\frac 13}t^{-\frac 32 } )\frac 1{\langle t-\frac \xi {k} \rangle }\vert k-l \vert^2\\
    &\lesssim  \left(t\frac 1{\langle t-\frac \xi {k} \rangle^{\frac 32}  }+  \langle \mu^{-\frac 13}t^{-\frac 32 }\rangle \right)\vert k-l \vert^2\\
    &\le \left(t\tfrac1{(\min(1,\mu^{\frac13 }t ))}\sqrt{\frac{\p_t m_{\frac1 2 }}{m_{\frac1 2 }} (k,\eta)}\sqrt{\frac{\p_t m_{\frac1 2 }}{m_{\frac1 2 }} (l,\xi)}+  \langle \mu^{-\frac 13}t^{-\frac 32 }\rangle \right)\vert k-l,\eta-\xi\vert^5 \\
    \\
    &\le \left((t+\mu^{-\frac 13 })\sqrt{\frac{\p_t m_{\frac1 2 }}{m_{\frac1 2 }} (k,\eta)}\sqrt{\frac{\p_t m_{\frac1 2 }}{m_{\frac1 2 }} (l,\xi)}+  \langle \mu^{-\frac 13}t^{-\frac 32 }\rangle \right)\vert k-l,\eta-\xi\vert^5\\
\end{align*}

Property iv) is a consequence of using the mean value theorem and $\p_\eta M_L^\theta =\tfrac 1 k \p_t M_L^\theta  $.

\end{proof}

\begin{proof}[Proof of Lemma \ref{lem:Mb}]
    Monotenicity and property (i) are clear by definition, property (ii) works the same as for the previous linear weights and therefore we only prove (iii). In the following we write $M$ for $M_{L,d,c_1}$. We WLOG assume that $k, l\neq 0$, since $M(t,0,\cdot )=M(t,1,\cdot )$. We split 
    \begin{align*}
         M(k,\eta) -M(l,\xi)&=M(k,\eta) -M(l,\eta)+M(l,\eta) -M(l,\xi),
    \end{align*}
    and estimate the commutator separately. By mean calue theorem there exists a $\tilde \eta \in [\eta,\xi]$ such that
\begin{align*}
    M(k,\eta) -M(k,\xi) &= \p_\eta M (k,\tilde \eta)(\eta-\xi) \\
    &= \frac {2d_2 }k \frac 1{1+(d_1 +d_2(t-\frac {\tilde \eta} k))^2 }(\eta-\xi)
\end{align*}
We use that $\tfrac 1{1+(d_1 +d_2(t-\frac {\tilde \eta} k))^2 }\approx \tfrac 1{1+(t-\frac {\tilde \eta} k)^2 }$ and so

\begin{align*}
   \frac {\vert l,\xi\vert }{t} \vert  M(k,\eta) -M(k,\xi)\vert 
    &\lesssim  \frac {\vert l,\xi\vert }{t k} \tfrac 1{1+(t-\frac {\tilde \eta} k)^2 }\vert k-l,\eta-\xi\vert\lesssim \vert k-l,\eta-\xi\vert^3
\end{align*}
which is consistent with (iii). Now we prove 
    \begin{align*}
        \frac {\vert l,\xi\vert}{t} \vert M (k,\eta) -M(l,\eta ) \vert \lesssim \left( 1+\mu^{-\frac 13 } \vert \ln(\mu)\vert^{1+r} \sqrt{\frac{\p_t m_0 }{m_0} (k,\eta ) }\sqrt{\frac{\p_t m_0 }{m_0} (l,\xi ) } \right) \langle  k-l, \eta-\xi \rangle^5.
    \end{align*}
     We distinguish between the tree cases ($0\in[k,l] $ and $t\not\in [ \frac \xi l,\frac \xi k ]$),  ($0\not\in[k,l] $ and $t\not\in [ \frac \xi l,\frac \xi k ]$) and $t \in [ \frac \xi l,\frac \xi k ]$. For $0\in[k,l] $ and $t\not\in [ \frac \xi l,\frac \xi k ]$, $k$ and $l$ have different signs and so 
\begin{align*}
    t\ge \min\left(\vert \tfrac \xi l \vert , \vert \tfrac \xi k \vert \right)
\end{align*}
Furthermore since $\vert k-l\vert \ge \vert k\vert, \vert l\vert$ we estimate 
\begin{align*}
        \frac {\vert l,\xi\vert}{t} \vert M (k,\eta) -M(l,\eta ) \vert \lesssim \vert k-l\vert^2. 
    \end{align*}
Let $0\not\in[k,l] $ and $t\not\in [ \frac \xi l,\frac \xi k ]$, we consider $M(k,\eta)$ as a (smooth) function of $k\in \R\setminus \{0\}$, then there exists a $n \in [k,l] $ such that
\begin{align*}
     M (k,\xi) -M(l,\xi )   &=\p_k M(n ,\xi)(k-l) 
\end{align*}
and so 
\begin{align*}
    \vert M (k,\xi) -M(l,\xi ) \vert   &\le \frac \xi {n^2 }\tfrac 1 {1+\left(d_1 +d_2 (t-\frac \xi{n} )\right)^2 } \vert k-l\vert . 
\end{align*}
Since $t\not\in [ \frac \xi l,\frac \xi k ]$ we obtain 
\begin{align*}
    \tfrac 1 {1+\left(d_1 +d_2 (t-\frac \xi{n} )\right)^2 }&\le \tfrac 1 {1+\left(d_1 +d_2 (t-\frac \xi{k} )\right)^2 }+\tfrac 1 {1+\left(d_1 +d_2 (t-\frac \xi{l} )\right)^2 }\\
    &\lesssim \tfrac 1 {1+(t-\frac \xi k )^2}+\tfrac 1 {1+(t-\frac \xi l )^2}
\end{align*}
and so we estimate 
\begin{align*}
    \frac {\vert l,\xi\vert}{t} \vert M (k,\xi) -M(l,\xi ) \vert   &\le \frac {\xi\vert l,\xi\vert } {t n^2 }\left(\tfrac 1 {1+(t-\frac \xi k )^2}+\tfrac 1 {1+(t-\frac \xi l )^2}\right) \vert k-l\vert \\
    &\le  \left(t^{-1}+\tfrac t {1+(t-\frac \xi k )^2}+\tfrac t {1+(t-\frac \xi l )^2}\right) \vert k-l\vert \\
    &\le \left(t^{-1}+\mu^{\frac 1 3 }+\tfrac t{\langle t \mu^{\frac 1 3 }\rangle } \sqrt{\frac {\p_tm_0 }{m_0 }(k,\eta ) }\sqrt{\frac {\p_tm_0 }{m_0 }(l,\xi ) } \right) \langle  k-l, \eta-\xi \rangle^6  . 
\end{align*}
For $t\in [ \frac \xi l,\frac \xi k ]$ we wlog assume that $\vert l\vert \le \vert k\vert  $. Let $\frac \xi {l^2} \ge 1$, then by $\langle t-\frac \xi l\rangle \le \langle \frac \xi k  - \frac \xi l\rangle $, we infer 
\begin{align*}
    1 \le  \tfrac \xi{l^2 }  \frac 1{\langle t-\frac \xi l \rangle }\vert k-l\vert^2. 
\end{align*}
If $\frac \xi {l^2} \le 1$, then $1\le \tfrac 2{\langle t-\frac \xi l \rangle }$. As before by $ \vert e^{\vert x\vert}-1\vert \le\vert x\vert e^{\vert x\vert}$
\begin{align*}
    \vert M (k,\xi) -M(l,\xi )\vert \lesssim  \vert \frac \xi l-\frac \xi k \vert \le  \tfrac \xi{l^2 }  \frac 1{\langle t-\frac \xi l \rangle }\vert k-l\vert^2.
\end{align*}
Using these estimates, we infer 
\begin{align*}
    \frac {\vert l,\xi\vert}{t} \vert M (k,\xi) -M(l,\xi ) \vert   &\le   \tfrac \xi{l^2 } \tfrac {\vert l,\xi\vert } {t  }\tfrac 1{\langle t-\frac \xi l \rangle }\vert k-l\vert \le t\tfrac 1{\langle t-\frac \xi l \rangle }\vert k-l\vert
\end{align*}
and so distinguishing between $\langle t-\frac \xi l \rangle\ge \mu^{-1}$ and $\langle t-\frac \xi l \rangle\le \mu^{-1}$  yields
\begin{align*}
    \frac {\vert l,\xi\vert}{t} \vert M (k,\xi) -M(l,\xi ) \vert   &\le t \left( \mu +\tfrac {\vert \ln(\mu)\vert^{1+r}}{\langle t-\frac \xi l \rangle \vert \ln(1+\langle t-\frac \xi l \rangle)\vert^{1+r}} \right) \vert k-l\vert \\
    &\lesssim  \left( t\mu^{\frac 13 } +(t+\mu^{-\frac 13 })\vert \ln(\mu)\vert^{1+r}\sqrt{\frac{\p_t m_0 }{m_0} (k,\eta ) }\sqrt{\frac{\p_t m_0 }{m_0} (l,\xi ) } \right) \langle  k-l, \eta-\xi \rangle^5 . 
\end{align*}

\end{proof}

\bibliographystyle{alpha} %alpha IEEEtran alphaabbr

\bibliography{library}

\end{document}